\documentclass{amsart}%
\usepackage{amsfonts, amsmath, amssymb}%
\usepackage{graphicx, epic, epsf, enumerate}
\usepackage{stmaryrd}

\input xy
\xyoption{all}

\newtheorem{theorem}{Theorem}
\theoremstyle{definition}

\newtheorem{example}{Example}
\newtheorem{lemma}{Lemma}

\newtheorem{remark}{Remark}

\numberwithin{equation}{section}
 
\newcommand{\Z}{\mathbb{Z}}
\newcommand{\N}{\mathbb{N}}
\def\P{{\mathcal P}}

\begin{document}

\title{A formula for the Dubrovnik polynomial of rational knots}

\author{Carmen Caprau}
\address{Department of Mathematics, California State University, Fresno, CA 93740, USA}
\email{ccaprau@csufresno.edu}
\urladdr{}
\author{Katherine Urabe}
\address{Department of Mathematics, California State University, Fresno, CA 93740, USA}
\email{kturabe@mail.fresnostate.edu}

\date{}

\begin{abstract}
We provide a formula for the Dubrovnik polynomial of a rational knot in terms of the entries of the tuple associated with a braid-form diagram of the knot. Our calculations can be easily carried out using a computer algebra system.

\end{abstract}
\maketitle

\section{Introduction}\label{sec:intro}

Rational knots and links are the simplest class of alternating links of one or two unknotted components. All knots and links up to ten crossings are either rational or are obtained by inserting  rational tangles into a small number of planar graphs (see~\cite{C}). 
Other names for rational knots are 2-bridge knots and 4-plats. The names rational knot and rational link were coined by John Conway who defined them as numerator closures of rational tangles, which form a basis for their classification. A rational tangle is the result of consecutive twists on neighboring endpoints of two trivial arcs. Rational knots and rational tangles have proved useful in the study of DNA recombination.

Throughout this paper, we refer to knots and links using the generic term `knots'. In~\cite{L}, Lu and Zhong provided and algorithm to compute the 2-variable Kauffman polynomial~\cite{Ka} of unoriented rational knots using Kauffman skein theory and linear algebra techniques. On the other hand, Duzhin and Shkolnikov~\cite{D} gave a formula for the HOMFLY-PT polynomial~\cite{HOMFLY, PT} of oriented rational knots in terms of a continued fraction for the rational number that represents the given knot.

A rational knot admits a diagram in braid form with $n$ sections of twists, from which we can associate an $n$-tuple to the given diagram. Using the properties of braid-form diagrams of rational knots and inspired by the approach in~\cite{D} (namely, deriving a reduction formula and associating to it a computational rooted tree), in this paper we provide a closed-form expression for the 2-variable Kauffman polynomial of a rational knot in terms of the entries in the $n$-tuple representing a braid-form diagram of the knot. We will work with the Dubrovnik version of the 2-variable Kauffman polynomial, called the Dubrovnik polynomial. Due to the nature of the skein relation defining the Dubrovnik polynomial (or the Kauffman polynomial, for that matter), deriving the desired closed-form for this polynomial of rational knots is more challenging than for the case of the HOMFLY-PT polynomial.

The paper is organized as follows: In Section~\ref{sec:rational-knots} we briefly review some properties about rational tangles and rational knots, which are needed for the purpose of this paper. In Section~\ref{sec:Dub} we look at the Dubrovnik polynomial of a rational knot diagram in braid-form and write it in terms of the polynomials associated with diagrams that are still in braid-form but which contain fewer twists. Our key reduction formulas are derived in Section~\ref{sec:reduction} and used in Section~\ref{closed-form} to obtain a closed-form expression that computes the Dubrovnik polynomial of a standard braid-form diagram of a rational knot in terms of the entries of the $n$-tuple associated with the given diagram. We finish with an appendix containing the Mathematica\textsuperscript{\textregistered} code, written by the second named author, that computes the Dubrovnik polynomial of a rational knot diagram in braid-form, based on the formulas obtained in Section~\ref{sec:Dub}.

The paper grew out of the second named author's master's thesis at California State University, Fresno. 
\section{Rational knots and tangles}\label{sec:rational-knots}

\textit{Rational tangles} are a special type of $2$-tangles that are obtained by applying a finite number of consecutive twists of neighboring endpoints starting from the two unknotted arcs $[0]$ or $[\infty]$ (called the \textit{trivial 2-tangles}) depicted in Figure \ref{bt}. An example of a rational tangle diagram is shown in Figure \ref{tsf}.

\begin{figure}[ht] 
\raisebox{-10pt}{\includegraphics[angle=90, height=0.4in]{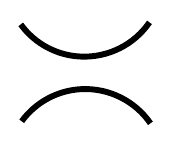}\hskip 30 pt  \includegraphics[height=0.4in]{sm}}
 \put(-80,-20){\fontsize{11}{10}$[\infty]$}
 \put(-18, -20){\fontsize{11}{10}$[0]$}
  \caption{The trivial 2-tangles $[0]$ and $[\infty]$} \label{bt}
\end{figure}

\begin{figure}[ht]
\[\includegraphics[scale=.15]{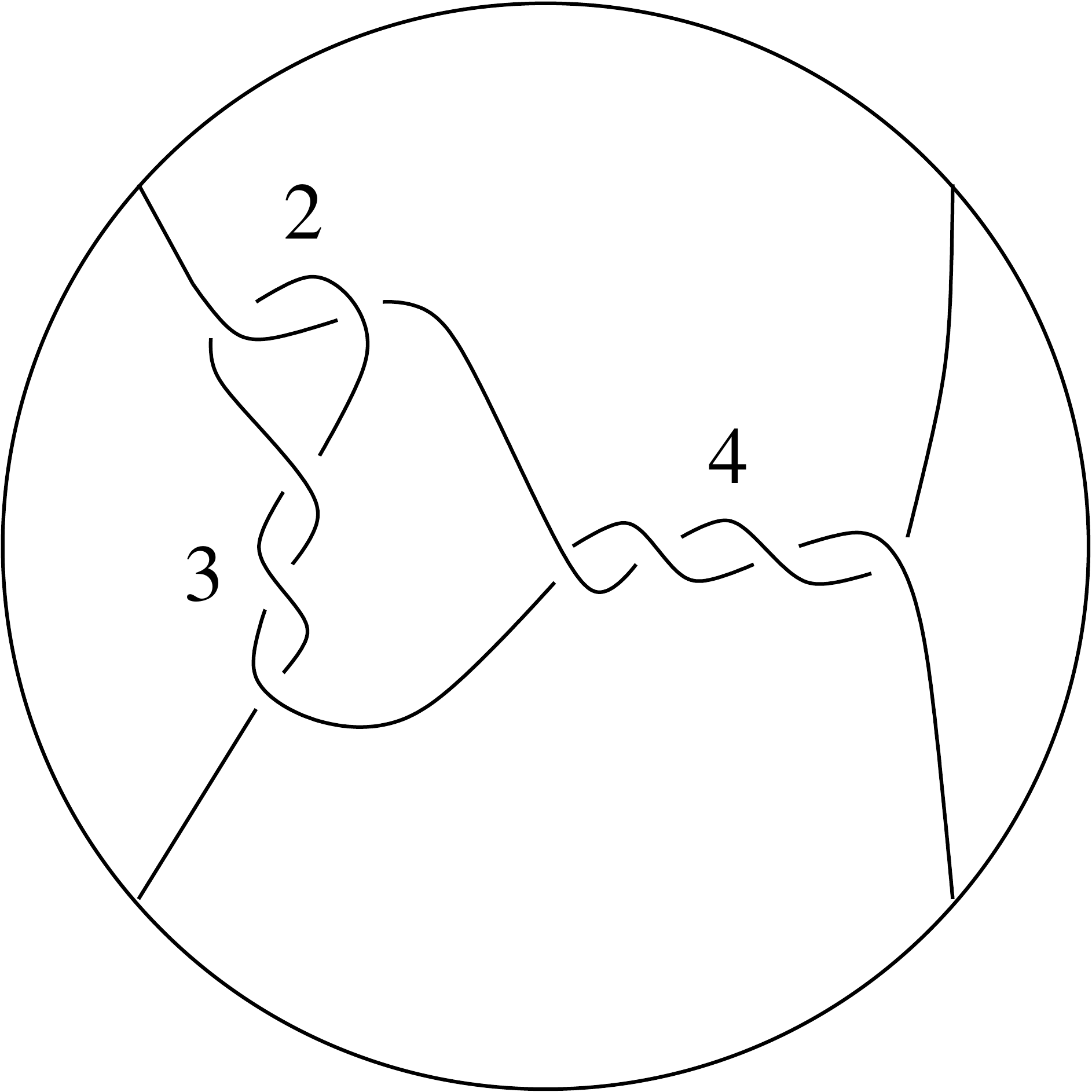}\]
\caption{The rational tangle $T(4, 3, 2)$ in standard form\label{tsf}}
\end{figure}

Taking the \textit{numerator closure}, $N(T)$, or the \textit{denominator closure}, $D(T)$, of a $2$-tangle $T$ (as shown in Figure \ref{closures}) results in a knot or a link. In fact, every knot or link can arise as the numerator closure of some 2-tangle (see \cite{KL2}). However, numerator (or denominator) closures of different rational tangles may result in the same knot. 

Numerator and denominator closures of rational tangles give rise to \textit{rational knots}. These are alternating knots with one or two components. It is an interesting fact that all knots and links up to ten crossings are either rational knots or are obtained from rational knots by inserting rational tangles into simple planar graphs. For readings on rational knots and rational tangles we refer the reader to~\cite{Bu, C, GK, KL1, KL2, Kaw, Mu, R, S}.

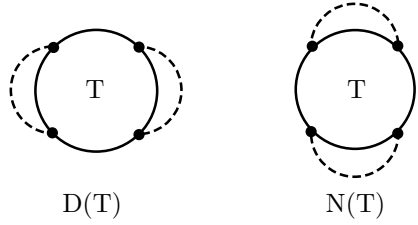
\begin{figure}[ht]
\[
\raisebox{-10pt}{\includegraphics[height=1in]{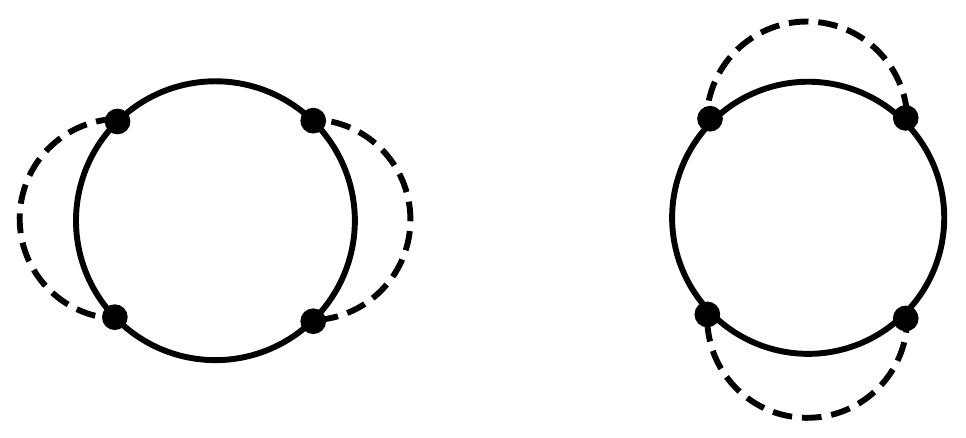}}
 \put(-37,-20){\fontsize{9}{10}N(T)}
  \put(-140, -20){\fontsize{9}{10} D(T)}
    \put(-128, 23){\fontsize{13}{10}T}
   \put(-29, 23){\fontsize{13}{10}T}
 \]
\caption{The denominator and numerator closures of a $2$-tangle $T$} \label{closures}
 \end{figure}
 
Conway~\cite{C} associated to a rational tangle diagram $T$ a unique, reduced rational number (or infinity), $F(T)$, called the \textit{fraction of the tangle}, and showed that two rational tangles are equivalent if and only if they have the same fraction. Specifically, for a rational tangle in standard form, $T(b_1, b_2, \dots, b_n)$, its fraction is calculated by the continued fraction
 \[F(T) = [b_1, b_2, \ldots, b_n]: = \displaystyle b_1+\cfrac{1}{b_2+\ldots+\cfrac{1}{b_{n-1}+\cfrac{1}{b_n}}}\] 
where $b_1 \in \mathbb{Z}$ and $b_2, \ldots, b_n\in \Z\backslash \{0\}$. Proofs of this statement can be found in~\cite{Bu, GK, KL1, Mo}.

A rational knot admits a diagram in braid form, as explained in Figures~\ref{braid-form-odd} and~\ref{braid-form-even}. We denote by $D[b_1, b_2, \dots, b_n]$ a \textit{standard braid-form diagram} (or shortly, \textit{standard diagram}) of a rational knot, where $b_i$'s are integers. 

\begin{figure}[ht]
\raisebox{-10pt}{\includegraphics[height=0.7in]{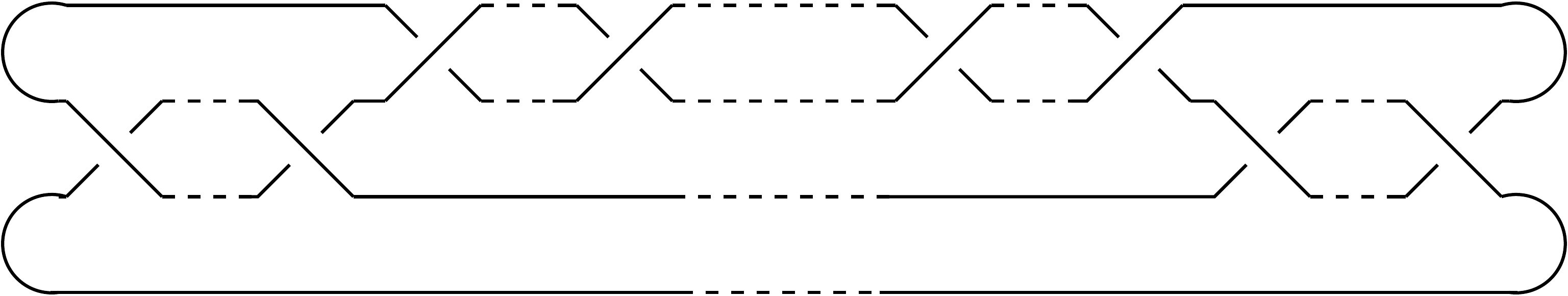}}
\put(-235,13){\fontsize{9}{10}$b_1$}
\put(-181,30){\fontsize{9}{10}$b_2$}
\put(-100,30){\fontsize{9}{10}$b_{n-1}$}
\put(-40,13){\fontsize{9}{10}$b_n$}
\caption{Standard braid-form diagram, $n$ odd }\label{braid-form-odd}
\end{figure}

\begin{figure}[ht]
\raisebox{-10pt}{\includegraphics[height=0.7in]{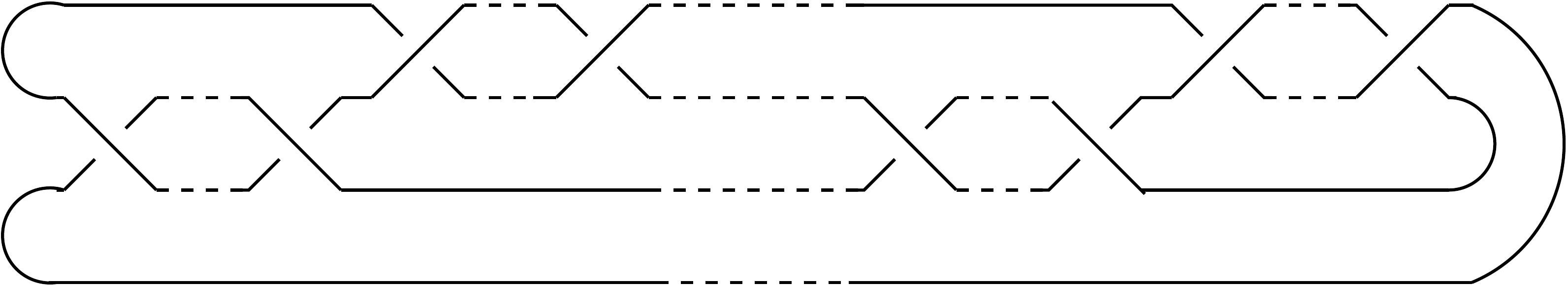}}
\put(-245,13){\fontsize{9}{10}$b_1$}
\put(-190,30){\fontsize{9}{10}$b_2$}
\put(-108,13){\fontsize{9}{10}$b_{n-1}$}
\put(-48,30){\fontsize{9}{10}$b_n$}
\caption{Standard braid-form diagram, $n$ even }\label{braid-form-even}
\end{figure}
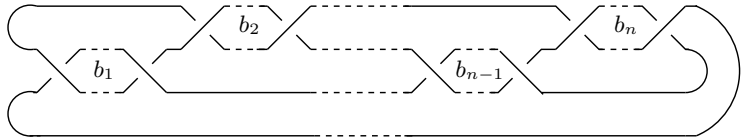

A standard diagram of a rational knot is obtained by taking a special closure of a 4-strand braid with $n$ sections of twists, where the number of half-twists in each section is denoted by the integer $|b_i|$ and the sign of $b_i$ is defined as follows: if $i$ is odd, then the left twist (Figure~\ref{twist}) is positive, and if $i$ is even, then the right twist is positive (equivalently, the left twist is negative for $i$ even). In Figures ~\ref{braid-form-odd} and~\ref{braid-form-even} all integers $b_i$ are positive. 

 Note that the special closure for $n$ even is the denominator closure of a rational tangle and for $n$ odd is the numerator closure.

\begin{figure}[ht]
\raisebox{-10pt}{\includegraphics[height=0.3in]{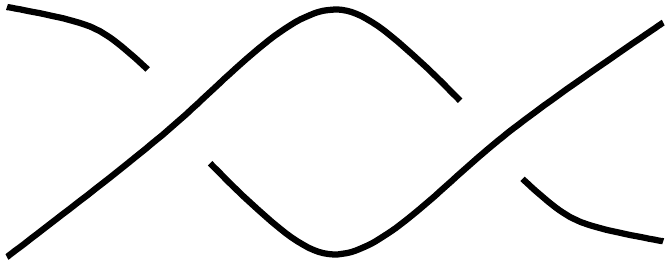}} \hspace{1cm} \reflectbox{\raisebox{-10pt}{\includegraphics[height=0.3in]{right-twist}}} 
\put(-140,-20){\fontsize{9}{10}right twist}
\put(-45,-20){\fontsize{9}{10}left twist}
\caption{}\label{twist}
\end{figure}

If a rational tangle $T(b_1, b_2, \dots, b_n)$ has its fraction a rational number other than $0$ or $\infty$, then we can always find $b_i$ such that the signs of all the $b_i$ are the same (see~\cite{Mu, KL2}). Hence, we assume that a standard diagram of a rational knot is an alternating diagram.

In addition, we can always assume that for a standard diagram $D[b_1, b_2, \dots, b_n]$ of a rational knot, $n$ is odd. This follows from the following properties of a continued fraction expansion $[b_1, b_2, \dots, b_n]$ for a rational number $\frac{p}{q}$ such that $b_i >0$:
\begin{itemize}
\item If $n$ is even and $b_n>1$, then
\[ [b_1, b_2, \dots, b_n] = [b_1, b_2, \dots, b_n-1, 1] \]
\item If $n$ is even and $b_n = 1$, then
\[ [b_1, b_2, \dots, b_n] = [b_1, b_2, \dots, b_{n-1}+1] \]
\end{itemize}

The statement below is well-known (see for example~\cite{Mu}).
\begin{lemma}\label{properties-rknots}
The following statements hold:
\begin{enumerate}
\item $D[-b_1, -b_2, \dots, -b_{2k+1}]$ is the mirror image of $D[b_1, b_2, \dots, b_{2k+1}]$.
\item $D[b_1, b_2, \dots, b_{2k+1}]$ is ambient isotopic to $D[b_{2k+1}, \dots, a_2, a_1]$.
\item $D[b_1, b_2, \dots, b_{2k+1}]$ is ambient isotopic to $N(T(b_1, b_2, \dots, b_{2k+1}))$.
\end{enumerate}
\end{lemma}

We denote by $K(p/q)$ a rational knot with standard diagram $D[b_1, b_2, \dots, b_{2k+1}]$, with $b_i \neq 0$, where
\[\frac{p}{q}= [b_1, b_2, \dots, b_{2k+1}], \hspace{0.5cm} \gcd(p, q) = 1 \,\, \text{and} \,\,  p >0.  \]
If all $b_i$ are positive then $q >0$, and if all $b_i$ are negative then $q <0$. The integer $p$ is odd for a knot and even for a two-component link. It is known that two rational knots $K(p/q)$ and $K(p'/q')$ are equivalent if and only if $p = p'$ and $q' \equiv q^{\pm 1} \pmod p$ (see~\cite{Bu,KL1,Mu, S}). 

\section{The Dubrovnik polynomial of rational knots}\label{sec:Dub}
In~\cite{Ka}, Kauffman constructed a 2-variable Laurent polynomial which is an invariant of regular isotopy for unoriented knots. In this paper we work with the Dubrovnik version of Kauffman's polynomial, called the \textit{Dubrovnik polynomial}. 

The Dubrovnik polynomial of a knot $K$, denoted by $P(K) : = P(K)(z, a)$, is uniquely determined by the following axioms: 
\begin{enumerate}
\item[1.] $P(K)=P(K')$ if $K$ and $K'$ are regular isotopic knots.
\item[2.] $P\left(\raisebox{-12pt}{\includegraphics[angle=90,height = 0.4in]{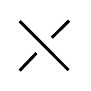}}\right)- P\left(\raisebox{-12pt}{\includegraphics[height = 0.4in]{neg}}\right) = z \left [P\left(\raisebox{-7pt}{\includegraphics[height = 0.3in, width = 0.4in]{sm}}\right) - P\left(\raisebox{-7pt}{\includegraphics[angle=90, height = 0.3in]{sm}}\right) \right]$. 
\item[3.] $P\left(\raisebox{-10pt}{\includegraphics[height = .35in]{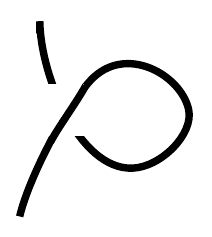}}\right)=a P\left(\raisebox{-10pt}{\includegraphics[height=.35in]{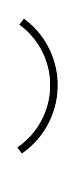}}\right)$ and $P\left(\reflectbox{\raisebox{-10pt}{\includegraphics[ height =.35in]{twista}}}\right)=a^{-1} P\left(\reflectbox{\raisebox{-10pt}{\includegraphics[scale=.5]{untwist}}}\right)$.
\item[4.] $P\left(\raisebox{-6pt}{\includegraphics[scale=.2]{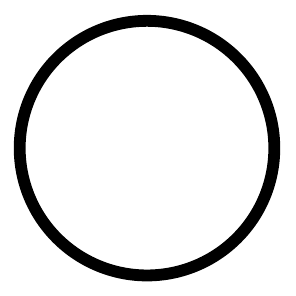}}\right)=1$.
\end{enumerate}
The diagrams in both sides of the second and third axioms above represent larger knot diagrams that are identical, except near a point where they differ as shown. 
We will use the following form for the second axiom;
\[ 
P\left(\raisebox{-12pt}{\includegraphics[angle=90,height = 0.4in]{neg}}\right) = P\left(\raisebox{-12pt}{\includegraphics[height = 0.4in]{neg}}\right) + z P\left(\raisebox{-7pt}{\includegraphics[height = 0.3in, width = 0.4in]{sm}}\right) -z P\left(\raisebox{-7pt}{\includegraphics[angle=90, height = 0.3in]{sm}}\right), 
\] 
and refer to it as the \textit{Dubrovnik skein relation}. Moreover, we will refer to the resulting three knot diagrams in the right-hand side of the Dubrovnik skein relation as the switched-crossing state, the A-state, and the B-state, respectively, of the given knot diagram.

The following statement is well-known and follows easily from the definition of the polynomial invariant $P$.

\begin{lemma}\label{mirror-image}
If $\overline{K}$ is the mirror image of $K$, then $P(\overline{K})(z, a) = P(K)(-z, a^{-1})$.
\end{lemma}

For more details about the Kauffman polynomial and the Dubrovnik version of it we refer the reader to~\cite{Ka, Ka1}. 

The goal of the paper is to give an algorithm which computes the Dubrovnik polynomial of a standard diagram $D[b_1, b_2, \dots, b_n]$ for a rational knot. We will use the following notation:
\[\P[b_1, b_2, \dots, b_n]:=  P(D[b_1, b_2, \dots, b_n]). \]

We will focus on the case with positive integers $b_i$. The case with negative integers $b_i$ follows from Lemma~\ref{properties-rknots} and Lemma~\ref{mirror-image}. For a standard diagram $D[b_1, b_2, \dots, b_n]$ of a rational knot, we call the integer $n$ the \textit{length} of the diagram.

We consider first a standard diagram of length three, $D[b_1, b_2, b_3]$, where $b_i\geq 3$ for $1\leq i\leq 3$. We obtain the tree diagram depicted in Figure~\ref{3tree}, whose edges are labeled by the weights of the polynomial evaluations of the resulting knot diagrams obtained by applying the Dubrovnik skein relation at the leftmost crossing in the original diagram. 
\begin{figure}[ht]
\raisebox{-10pt}{\includegraphics[height =1.6in]{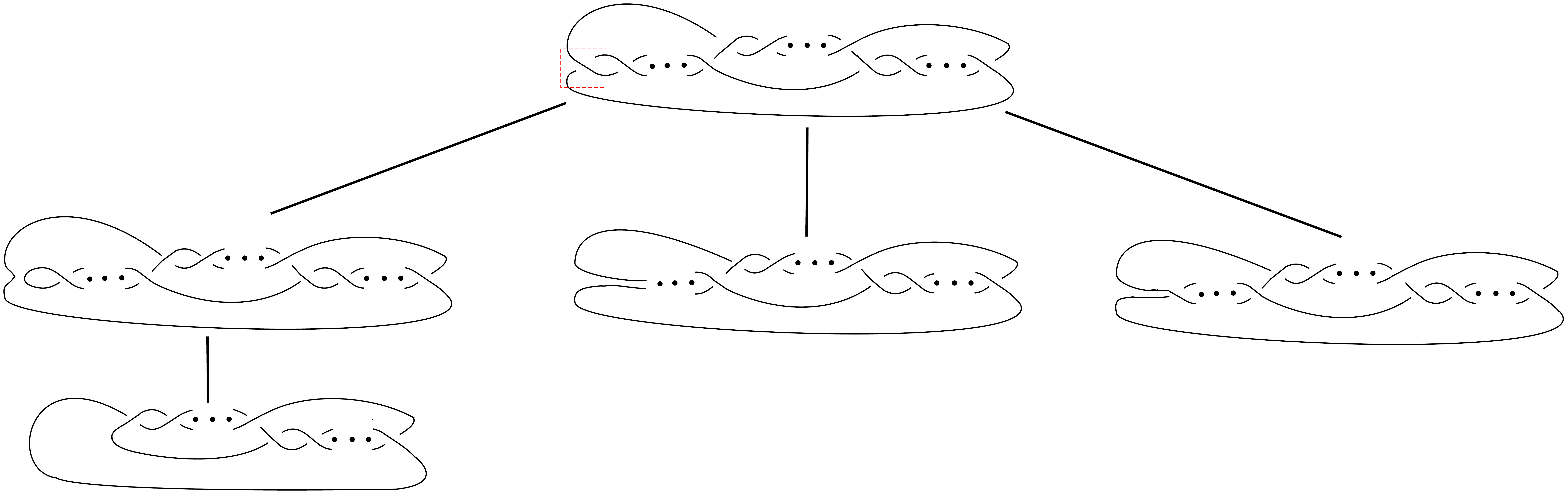}}
 \put(-100,42){\fontsize{9}{10}$b_1-1$}
 \put(-220,95){\fontsize{9}{10}$b_1$}
  \put(-360,45){\fontsize{9}{10}$b_1-1$}
  \put(-90, 70){\fontsize{9}{10}$z$}
    \put(-225, 45){\fontsize{9}{10}$ b_1-2$}
 \put(-282, 70){\fontsize{9}{10}$-z$}
    \put(-313, 17){\fontsize{9}{10}$a^{-b_1+1}$}
\caption{A tree diagram for $D[b_1, b_2, b_3]$} \label{3tree}
 \end{figure}
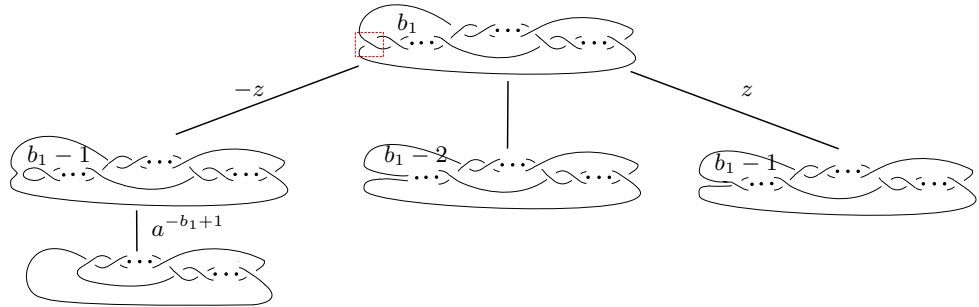

Note that the middle leaf of the tree in Figure~\ref{3tree} is obtained after applying a type II Reidemeister move. Moreover, the diagram at the bottom of the left-hand branch of the tree can be modified to a standard braid-form diagram, as exemplified in Figure~\ref{3treetrans}.
\begin{figure}[ht]
\raisebox{-10pt}{\includegraphics[scale=.20]{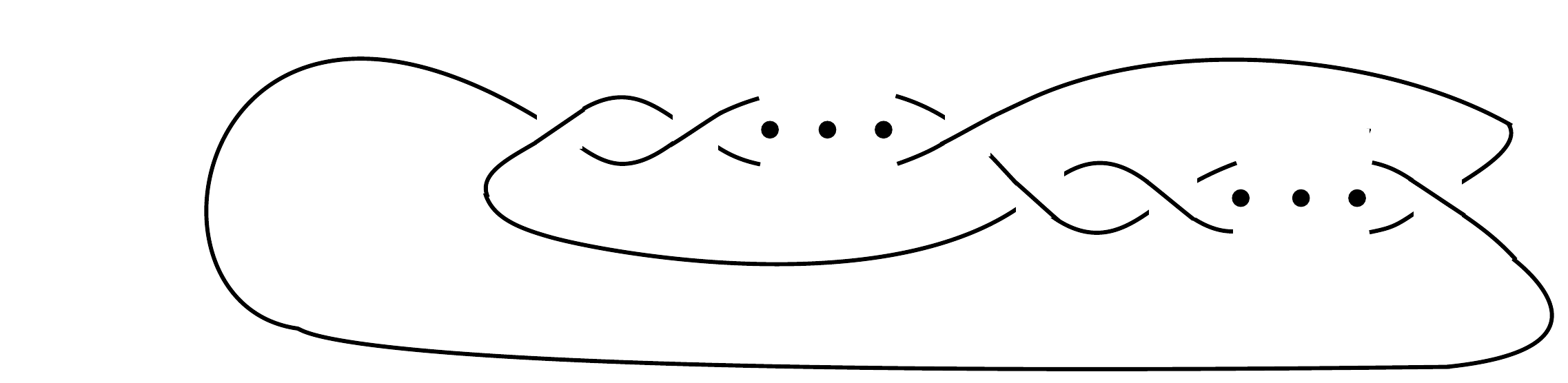} \hspace{35pt}\includegraphics[scale=.20]{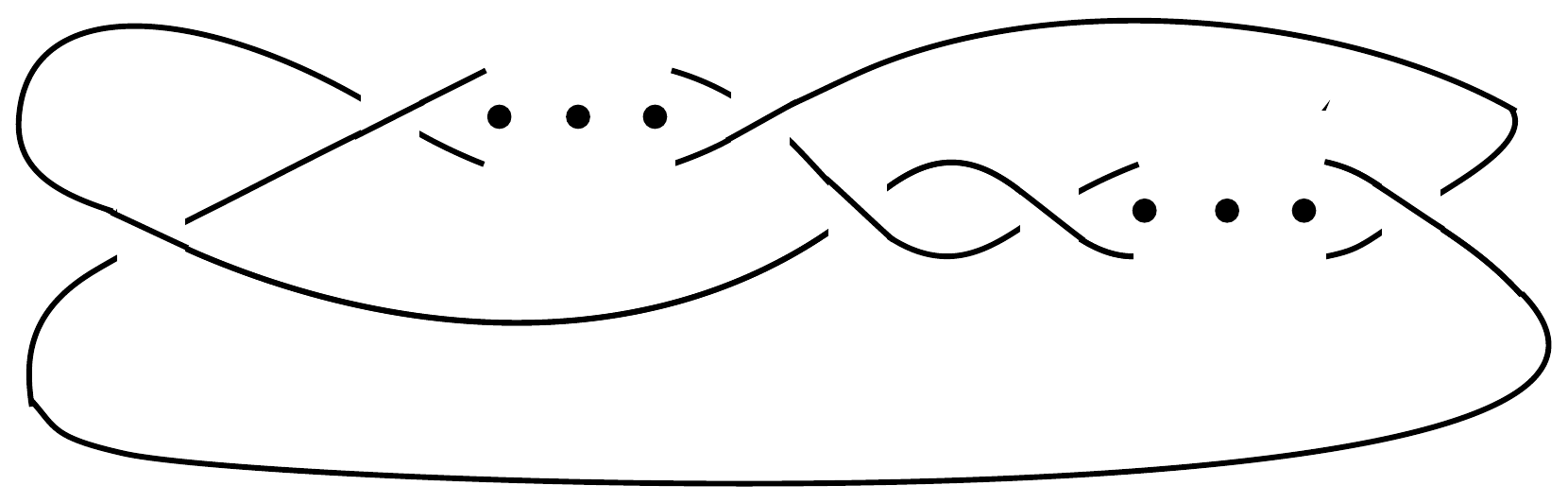}}
 \put(-210,20){\fontsize{9}{10}$ b_2$}
 \put(-140, 3){\fontsize{9}{10}$\longrightarrow $}
   \put(-75, 21){\fontsize{9}{10}$ b_2-1$}
  \put(-87, 10){\fontsize{9}{10}$ 1$}
\caption{Obtaining a standard braid-form diagram}\label{3treetrans}
\end{figure}
We obtain the following recursive relation for $\P[b_1, b_2, b_3]$, where $b_i\geq 3$ for $1\leq i\leq 3$:
\[  \P[ {b_1, b_2, b_3}]=
  \P[ {b_1 - 2, b_2, b_3}] - 
   z a^{-b_1 + 1} \P[ {1, b_2 - 1, b_3}]+z \P[ {b_1 - 1, b_2, b_3}].
\]
These types of recursive formulas form the foundation of our later work. In general, as in the case above, the switched-crossing state reduces a section by 2 half-twists, the A-state reduces a section by one half-twist, and the B-state reduces the diagram by one section of twists while contributing some power of $a$. Because of this, the remainder of our cases will deal with standard braid-form diagrams that have 1 or 2 half-twists in the first section of twists (corresponding to $b_1$), since these relations will completely reduce the first section. In addition, the move shown in Figure~\ref{3treetrans} will be commonly referred to as the \textit{sliding move} and not shown in detail for the other cases. 

Consider now a standard diagram $D[1,1,b_3]$ and the tree given in Figure \ref{1tree}.
\begin{figure}[ht]
\raisebox{-10pt}{\includegraphics[height=1.5in]{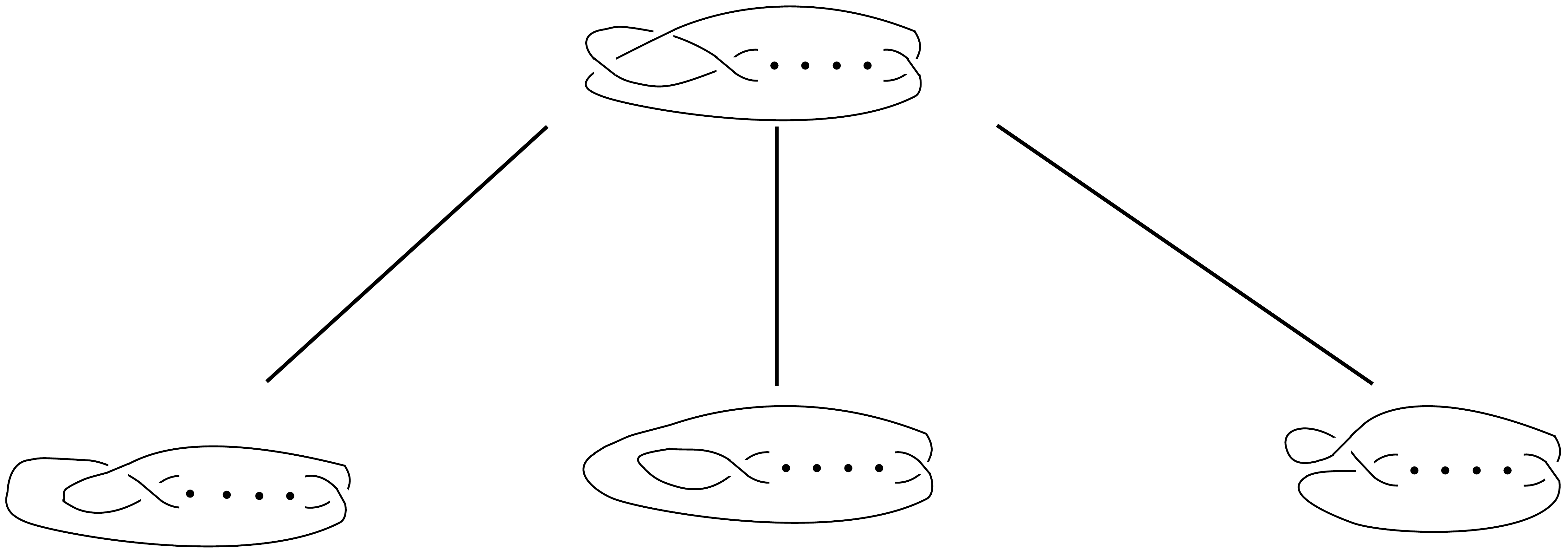}}
  \put(-244, 49){\fontsize{9}{10} $-z$}
    \put(-70, 49){\fontsize{9}{10}$z$}
\caption{A tree diagram for $D[1,1,b_3]$ \label{1tree}}
 \end{figure}

Observe that we applied a type II Reidemeister move to arrive at the diagram representing the middle leaf of the tree. We easily see that the Dubrovnik polynomial for the diagram $D[1,1,b_3]$ satisfies the following relation:
\[\P[ {1, 1, b_3}]  = a^{-b_3} - z \P[ {b_3 + 1}] +z a \P[ {b_3}].\]

Next we consider the standard braid-form diagram $D[2,b_2,b_3]$.
Using the tree for the diagram $D[2,b_2,b_3]$ depicted in Figure \ref{2tree} and the sliding move for the diagram representing the leftmost leaf in the tree, we obtain the following polynomial expression:
 \[ \P[ {2, b_2, b_3}] =  a^{b_2} \P[{b_3}] -  z a^{-1} \P[ {1, b_2 - 1, b_3}] +z  \P[ {1, b_2, b_3}].\]

\begin{figure}[ht]
\raisebox{-10pt}{\includegraphics[height=1.3in]{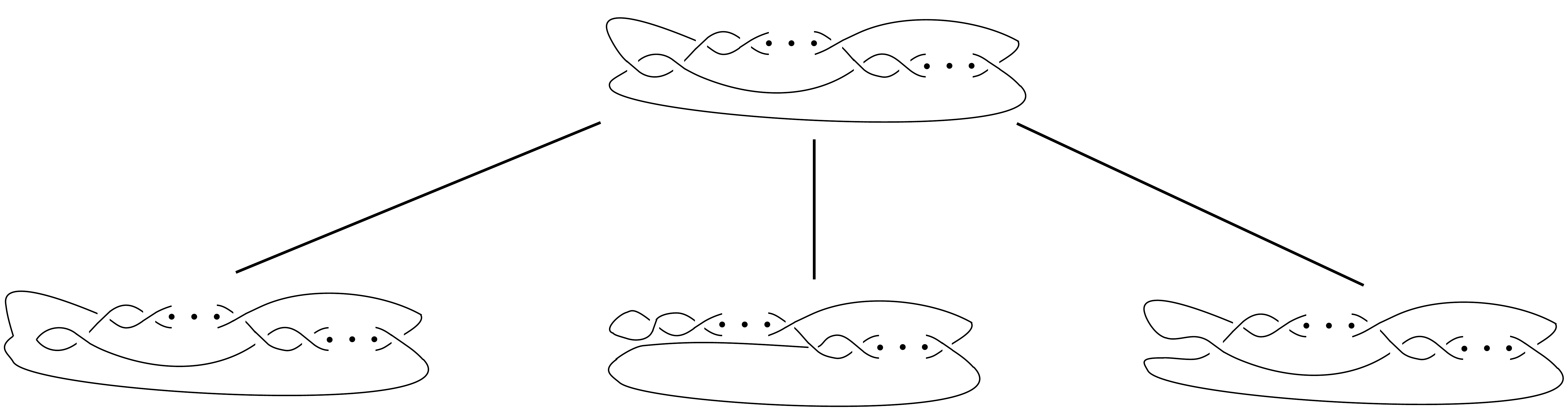}}
  \put(-282, 42){\fontsize{9}{10} $-z$}
    \put(-90, 42){\fontsize{9}{10} $z$}
\caption{A tree diagram for $D[2, b_2 ,b_3]$ \label{2tree}}
 \end{figure}

Lastly, consider the tree for the standard braid-form diagram $D[1,b_2,b_3]$ given in Figure \ref{other1}. By applying the sliding move to the first two leaves in the tree, the Dubrovnik polynomial of $D[1,b_2,b_3]$ satisfies the expression given below:
\[\P[ {1, b_2, b_3}] = \P[{1, b_2 - 2, b_3}] - z \P[ {1, b_2 - 1, b_3}] +z a^{b_2} \P[{b_3}]).\]

\begin{figure}[ht]
\raisebox{-10pt}{\includegraphics[height=1.2in]{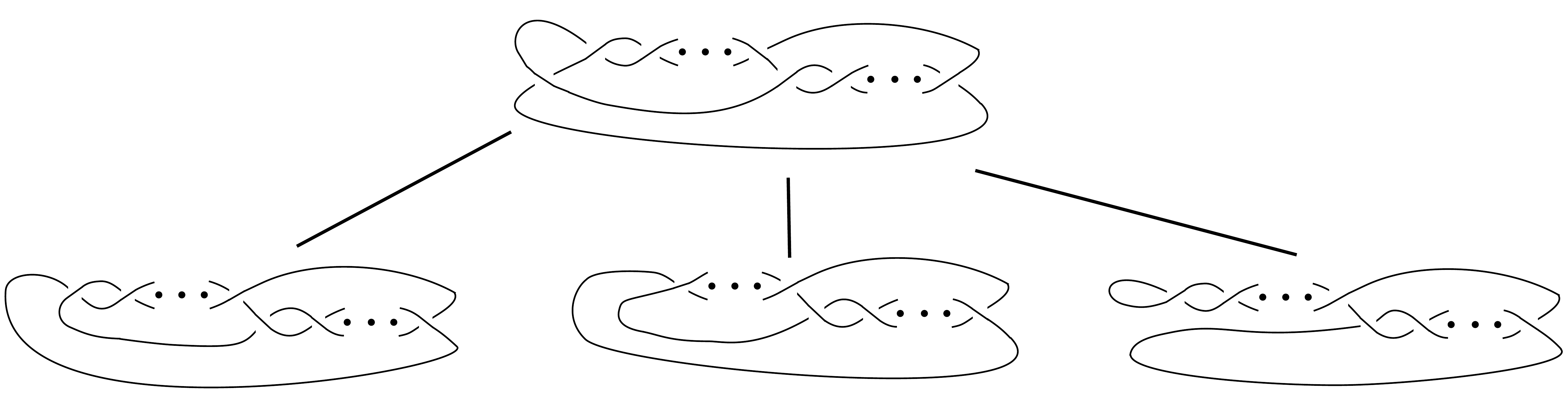}}
 \put(-277, 12){\fontsize{9}{10}$ b_3$}
 \put(-315, 18){\fontsize{9}{10}$ b_2$}
\put(-77,17){\fontsize{9}{10}$b_2$}
  \put(-27,12){\fontsize{9}{10}$b_3$}
\put(-93,33){\fontsize{9}{10}$z$}
  \put(-273, 37){\fontsize{9}{10}$ -z$}
\put(-205,71){\fontsize{9}{10}$b_2$}
  \put(-160, 65){\fontsize{9}{10}$ b_3$}
\put(-200,21){\fontsize{9}{10}$b_2-1$}
  \put(-149, 14){\fontsize{9}{10}$ b_3$}
\caption{A tree diagram for $D[1, b_2 ,b_3]$ \label{other1}}
 \end{figure}
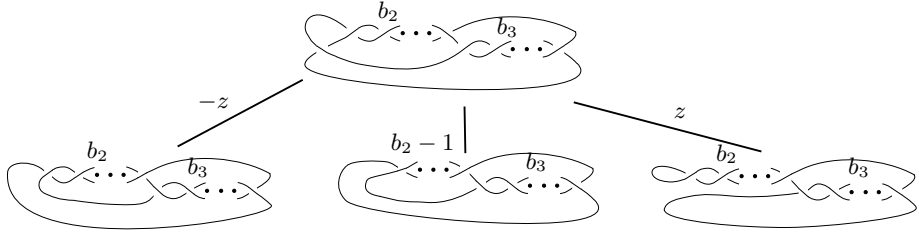

There are a few more cases that need to be considered, namely when $b_1$ or $b_2$ are equal to 1 or 2. These cases are treated in a similar way as above. We collect these cases and those shown above in the following statement. 

\begin{lemma}\label{algorithm}
The Dubrovnik polynomial of a standard braid-form diagram of length 3 with positive twists satisfies the following relations:
\begin{itemize}
\item[i.] $ \P[ {b_1, b_2, b_3}] =\P[ {b_1 - 2, b_2, b_3}] -z a^{1- b_1} \P[ {1, b_2 - 1, b_3}] + z \P[ {b_1 - 1, b_2, b_3}]$, for $b_1 \geq 3, b_2 \geq 2$.
\item[ii.]$ \P[ {2, b_2, b_3}] =  a^{b_2} \P[{b_3}] - z a^{-1} \P[ {1, b_2 - 1, b_3}] +z \P[ {1, b_2, b_3}]$, for $b_2 \geq 2$.
\item[iii.] $\P[ {1, b_2, b_3}] = \P[{1, b_2 - 2, b_3}] - z \P[ {1, b_2 - 1, b_3}] +z a^{b_2} \P[{b_3}]$, for $b_2 \geq 3$.
\item[iv.] $ \P[ {b_1, 1, b_3}] =  \P[ {b_1 - 2, 1, b_3}] - z a^{1 - b_1} \P[ {b_3 + 1}] +z \P[{b_1 - 1, 1, b_3}]$, for $b_1 \geq 3$.
\item[v.] $ \P[ {2, 1, b_3}] = a \P[{b_3}] - z a^{-1} \P[ {b_3 + 1}] +z \P[ {1, 1, b_3}]$.
\item[vi.] $\P[ {1, 2, b_3}]  =  \P[ {b_3 + 1}]- z \P[ {1, 1, b_3}] +z a^2 \P[ {b_3}]$.
\item[vii.] $\P[ {1, 1, b_3}]  = a^{-b_3} - z \P[ {b_3 + 1}] +z a \P[ {b_3}]$.
\end{itemize}
\end{lemma}

Note that the algorithm that allowed us to arrive at Lemma~\ref{algorithm} involved only the two leftmost groups of twists in a diagram of length 3. Therefore, the above cases can be generalized to a standard braid-form diagram of any odd length $n$ with positive half-twists $b_i$.

The one extra case we need to consider before we generalize the statement in Lemma~\ref{algorithm} is the standard diagram of length 5, $D[1,1,b_3,1,b_5]$, shown in Figure \ref{case111}.
\begin{figure}[ht]
\raisebox{-10pt}{\includegraphics[scale=.35]{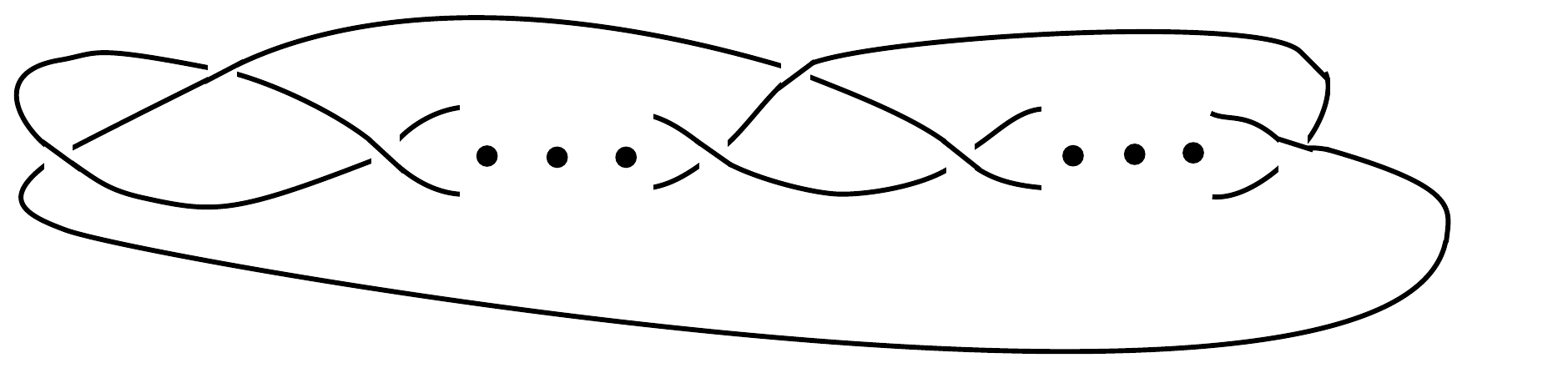}}
 \put(-170,30){\fontsize{9}{10}$ 1$}
  \put(-55, 25){\fontsize{9}{10}$ b_5$}
    \put(-183, 20){\fontsize{9}{10}$ 1$}
  \put(-95, 30){\fontsize{9}{10}$ 1$}
    \put(-124, 22){\fontsize{9}{10}$ b_3$}
\caption{Standard diagram $D[1,1, b_3, 1, b_5]$}\label{case111}
 \end{figure}
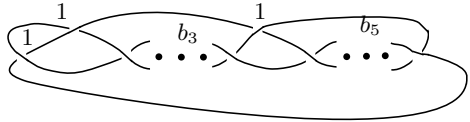
 
The Dubrovnik skein relation applied to the leftmost crossing in the diagram $D[1,1,b_3,1,b_5]$ yields the following expression:
 \[\P[{1, 1, b_3, 1, b_5}]=
   a^{-b_3} \P[{1 + b_5}] - z \P[{b_3 + 1, 1, b_5}] +z a \P[{b_3, 1, b_5}].\]

\begin{theorem}\label{algorithm-thm}
The Dubrovnik polynomial of a standard braid-form diagram $D[b_1, b_2, \ldots, b_n]$ with $n$ odd and positive twists satisfies the following relations:

\begin{eqnarray*} 
\P[b_1, b_2, \ldots, b_n] &=& \P[b_1 - 2, b_2,\ldots, b_n] -  z a^{1 -b_1} \P[1, b_2 - 1, \ldots, b_n] \\
 &&  +z \P[b_1 - 1, b_2, \ldots, b_n], \,\,\, \text{for} \,\, \,b_1 \geq 3, b_2 \geq 2.\\
\P[2, b_2, \ldots, b_n] &=& a^{b_2} \P[b_3\ldots, b_n] - z a^{-1} \P[1, b_2 - 1,\ldots, b_n] \\
&&+z \P[1, b_2, \ldots, b_n], \,\,\, \text{for}\,\,\, b_2 \geq 2.\\
 \P[1, b_2, \ldots, b_n] &=& \P[1, b_2 - 2, \ldots, b_n] -   z \P[1, b_2 - 1, \ldots, b_n] \\
 &&+z a^{b_2} \P[b_3\ldots, b_n], \,\,\, \text{for} \,\,\, b_2 \geq 3.\\
 \P[b_1, 1, b_3, \ldots, b_n] &=&  \P[b_1 - 2, 1, b_3,\ldots, b_n] -  z a^{1 - b_1} \P[b_3 + 1, \ldots, b_n]\\
 && +z \P[b_1 - 1, 1, b_3, \ldots, b_n] \,\,\, \text{for}\,\,\, b_1 \geq 3.\\
\P[2, 1, b_3, \ldots, b_n] &=& a \P[b_3, \ldots, b_n] -   z a^{-1} \P[b_3 + 1, \ldots, b_n] +z  \P[1, 1, b_3, \ldots, b_n]\\
\P[1, 2, b_3, \ldots, b_n] &=& \P[b_3 + 1,\ldots, b_n] -  z \P[1, 1, b_3, \ldots, b_n] +z a^2 \P[b_3, \ldots, b_n]\\
\P[1, 1, b_3, b_4, \ldots, b_n] &=& a^{-b_3} \P[1, b_4 - 1, \ldots, b_n] - z \P[b_3 + 1, b_4, \ldots, b_n]\\
&& +z a \P[b_3, b_4, \ldots, b_n] \,\,\, \text{for} \,\,\, b_4 \geq 2.\\
\P[1, 1, b_3, 1, b_5, \ldots, b_n] &=& a^{-b_3} \P[1 + b_5, \ldots, b_n] -z \P[b_3 + 1, 1, b_5, \ldots, b_n] \\
&&+za \P[b_3, 1, b_5, \ldots, b_n].
\end{eqnarray*}

\end{theorem}

\begin{proof}
The statement follows from Lemma~\ref{algorithm} and the discussion following it.
\end{proof}

Note that the standard braid-form diagrams appearing in both sides of any of the relations in Theorem~\ref{algorithm-thm} have the same parity.
From the patterns of these relations, the second named author wrote a program in Mathematica\textsuperscript{\textregistered} (see Appendix) which computes the Dubrovnik polynomial of any rational knot from a standard braid-form diagram. 

It is worth noting that although we began the reduction algorithm at the leftmost section of twists for programming purposes, beginning the reduction at the right hand side of the braid-form diagram reduces the number of cases needed to be considered. This is because the Dubrovnik skein relation, when applied to the right hand side of the diagram, does not result in diagrams (states) which are not in braid form, and thus does not require the sliding move. 

\section{Coefficient polynomials and a reduction formula} \label{sec:reduction}

In this section we use the mechanics of the Dubrovnik skein relation described in Section~\ref{sec:Dub}, to create an expression for the Dubrovnik polynomial of a standard braid-form diagram relative to the number of half-twists in a particular section. 

Our approach is motivated by the consistent recurrence of what we will call the \textit{coefficient polynomials} of the Dubrovnik polynomial for the reduced diagrams. 

To begin with, we borrow a notation from~\cite{D} and consider a family of links $L_m$ (where $m$ is an integer) which are identical except within a certain ball, where they have the segment indicated as in Figure \ref{fln}. Thus we are considering links that are identical except for a chosen section of twists (see Figure \ref{ln} for an example).

\begin{figure}[ht]
\raisebox{-10pt}{\includegraphics[width=4.5in]{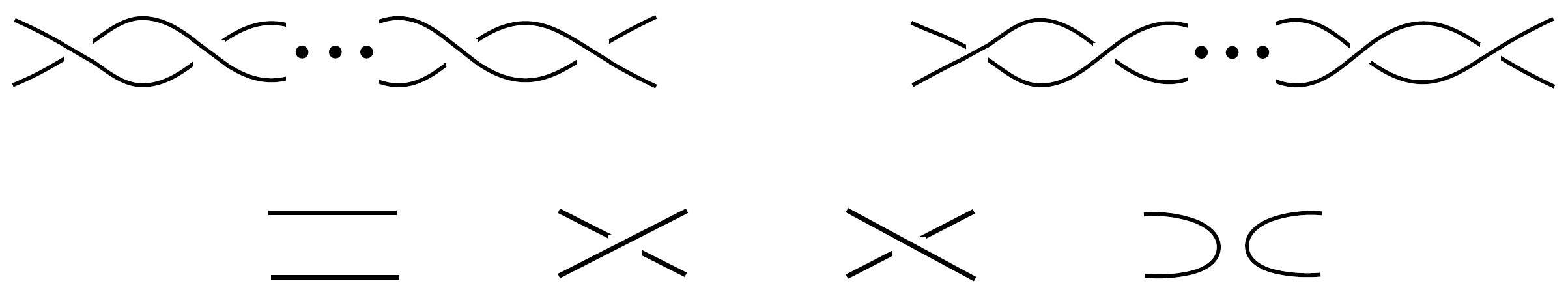}}
 \put(-285,20){\fontsize{10}{10} $L_m, \,\, m>0$}
  \put(-90, 20){\fontsize{10}{10}$L_m, \,\, m<0$}
    \put(-265, -26){\fontsize{10}{10} $ L_0$}
  \put(-205, -26){\fontsize{10}{10} $ L_{-}$}
  \put(-145, -26){\fontsize{10}{10} $L _{+}$}
  \put(-80, -26){\fontsize{10}{10} $ L_{\infty}$}
\caption{The family of links $L_m$ \label{fln}}
 \end{figure}
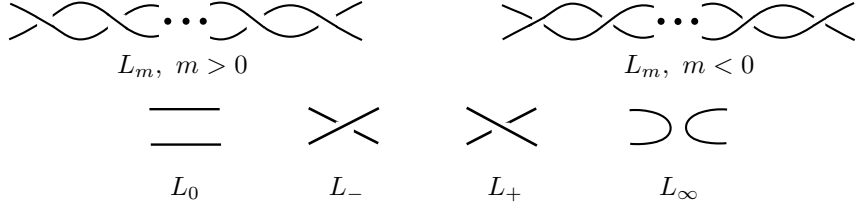

 For our purposes, $L_-$ (or $L_{+}$) is used in conjunction with $L_m$ for $m>0$ (or $m<0$). We will explicitly show the case of $D[b_1, \ldots, b_n]$ for all $b_i$'s positive (the negative case is treated similarly). The diagram $L_+$ or $L_ -$ appears in our reduction formula because of the possibility that we could have a section with one half-twist (unlike in \cite{D}, where all $b_i$'s are even). As discussed previously, applying the Dubrovnik skein relation to a section of one half-twist results in changes to the next section of twists, which $L_+$ and $L_-$ address.
 
 We begin by considering a section of $m >0$ half-twists in the rightmost block of twists in a standard diagram $D[b_1, \ldots, b_n]$, where $n$ is odd and all $b_i$'s are positive. 
 Figure \ref{ln} shows a family $L_m$ for $m = 3$.
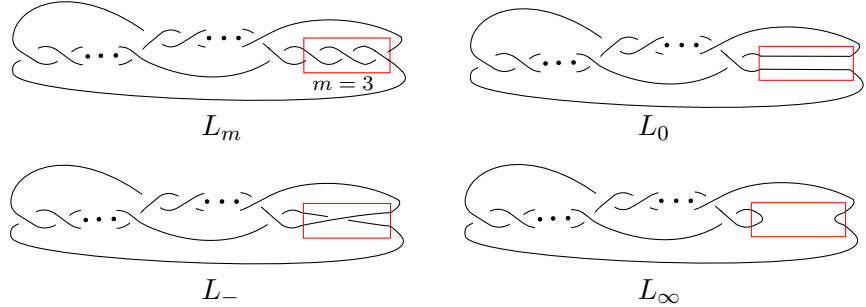
\begin{figure}[ht]
\raisebox{-10pt}{\includegraphics[width=4.5in]{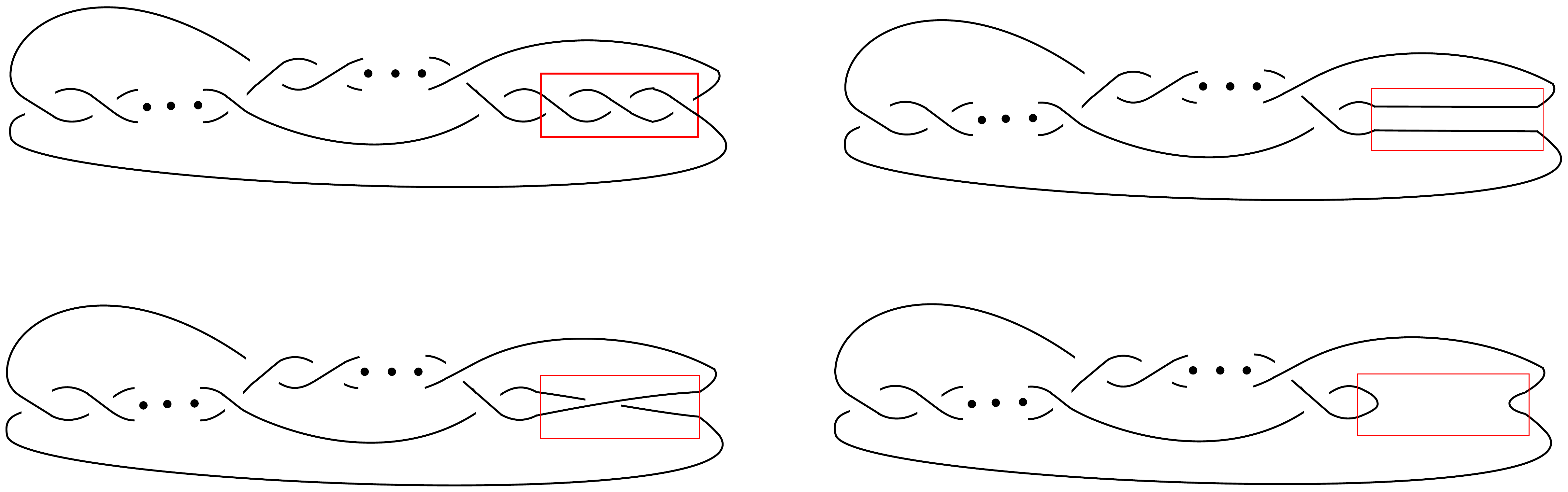}}
  \put(-90, 40){\fontsize{11}{10} $L_0$}
    \put(-255, -22){\fontsize{11}{10} $ L_-$}
  \put(-90, -22){\fontsize{11}{10} $ L_\infty$}
 \put(-255, 40){\fontsize{11}{10} $ L_m$}
  \put(-213, 57){\fontsize{8}{10} $m =3$}
\caption{An example for $L_m$, $m=3$ \label{ln}}
 \end{figure}
 
 \begin{remark}\label{relation} 
 
 In Section~\ref{sec:Dub} we discussed how applying the Dubrovnik skein relation affects a diagram. Observe that since the A-state reduces a diagram by one half-twist, it is equivalent to $L_{m-1}$, i.e., it has $m-1$ half-twists in the chosen section. In addition, the switched-crossing state reduces the diagram by two half-twists, and therefore, it is equivalent to $L_{m-2}$. The trickier part is the relationship of the B-state with $L_m$. The Dubrovnik skein relation applied to the rightmost crossing in a standard diagram $D[b_1, \ldots, b_n]$ results in a B-state whose polynomial evaluation is $a^{1 - b_n}\P[b_1, \ldots, b_{n-1}]$ for $n$ odd and $a^{b_n-1}\P[b_1, \ldots, b_{n-1}]$ for $n$ even. On the other hand, the polynomial evaluation, $P(L_{\infty})$, of the diagram $L_{\infty}$ is $a^{m-b_n}\P[b_1, \ldots, b_{n-1}]$ if $n$ is odd and $a^{b_n-m}\P[b_1, \ldots, b_{n-1}]$ if $n$ is even. Therefore the evaluation of the B-state is $a^{1-m}P(L_{\infty})$ for $n$ odd and $a^{m-1}P(L_{\infty})$ for $n$ even. Therefore, the Dubrovnik skein relation applied at the rightmost crossing in the diagram $D[b_1, \ldots, b_n]$ with $n$ odd can be rewritten as 
 \begin{eqnarray}\label{eq:first-red}
 P(L_m)= P(L_{m-2})-za^{1-m}P(L_{\infty})+zP(L_{m-1}).
 \end{eqnarray}
  \end{remark}
 
 \subsection{Recurrence relations for the coefficient polynomials}
  
 Let $A_m$ denote the coefficient of $P(L_-)$, $B_m$ the coefficient of $P(L_0)$, and $C_m$ the coefficient of $P(L_\infty)$. We will show that we have the following recurrence relations: 
\[A_m=zA_{m-1}+A_{m-2},  \hspace{15pt} B_m=zB_{m-1}+B_{m-2}, \hspace{7pt} \text{and}\hspace{7pt}  C_m=a^{-1}C_{m-1}+zB_{m-1},\] 
where $m >2$.
Before that, we prove three lemmas involving these coefficients.

\begin{lemma}\label{b}
 Let $B_m(z) =\displaystyle\sum_{i=0}^{\lfloor{\frac{m}{2}}\rfloor}z^{m-2i}\binom{m-i}{i}$, for $m \in \N$. Then 
 
\[B_m=zB_{m-1}+B_{m-2}, \,\,\text{for all} \,\,  m >2.\]
\end{lemma}
\begin{proof} Let $B_m$ be as above, where $m>2$. Then, we have:
\begin{eqnarray*}
zB_{m-1}+B_{m-2}&=&z\sum_{i=0}^{\lfloor{\frac{m-1}{2}}\rfloor}z^{m-1-2i}\binom{m-1-i}{i}+\sum_{i=0}^{\lfloor{\frac{m-2}{2}}\rfloor}z^{m-2-2i}\binom{m-2-i}{i}\\
&=&\sum_{i=0}^{\lfloor{\frac{m-1}{2}}\rfloor}z^{m-2i}\binom{m-1-i}{i}+\sum_{i=1}^{\lfloor{\frac{m}{2}}\rfloor}z^{m-2i}\binom{m-1-i}{i-1}\\
&=&\sum_{i=1}^{\lfloor{\frac{m-1}{2}}\rfloor}z^{m-2i}\binom{m-1-i}{i}+\sum_{i=1}^{\lfloor{\frac{m}{2}}\rfloor}z^{m-2i}\binom{m-1-i}{i-1}+z^m.
\end{eqnarray*}
Case 1: $m$ odd. Then $\lfloor{\frac{m-1}{2}}\rfloor=\lfloor{\frac{m}{2}}\rfloor$.
Thus we have
\begin{eqnarray*}
&&\sum_{i=1}^{\lfloor{\frac{m}{2}}\rfloor}z^{m-2i}\binom{m-1-i}{i}+\sum_{i=1}^{\lfloor{\frac{m}{2}}\rfloor}z^{m-2i}\binom{m-1-i}{i-1}+z^m\\
&=&\sum_{i=1}^{\lfloor{\frac{m}{2}}\rfloor}z^{m-2i}\left[\binom{m-1-i}{i}+\binom{m-1-i}{i-1}\right]  +z^m\\
&=&\sum_{i=1}^{\lfloor{\frac{m}{2}}\rfloor}z^{m-2i}\binom{m-i}{i}+z^m\\
&=&\sum_{i=0}^{\lfloor{\frac{m}{2}}\rfloor}z^{m-2i}\binom{m-i}{i}=B_m.
\end{eqnarray*}
Case 2: $m$ even. Then $\lfloor{\frac{m-1}{2}}\rfloor=\lfloor{\frac{m}{2}}\rfloor-1$, and we have:
\begin{eqnarray*}
&&\sum_{i=1}^{\lfloor{\frac{m-1}{2}}\rfloor}z^{m-2i}\binom{m-1-i}{i}+\sum_{i=1}^{\lfloor{\frac{m}{2}}\rfloor}z^{m-2i}\binom{m-1-i}{i-1}+z^m\\
&=&\sum_{i=1}^{\lfloor{\frac{m}{2}}\rfloor-1}z^{m-2i}\left[\binom{m-1-i}{i}+\binom{m-1-i}{i-1}\right]+z^m\\&&+z^{m-2\lfloor{\frac{m}{2}}\rfloor}\binom{m-\lfloor{\frac{m}{2}}\rfloor-1}{\lfloor{\frac{m}{2}}\rfloor-1}\\
&=&\sum_{i=1}^{\lfloor{\frac{m}{2}}\rfloor-1}z^{m-2i}\binom{m-i}{i}+z^m+z^{m-2\lfloor{\frac{m}{2}}\rfloor}\binom{m-\lfloor{\frac{m}{2}}\rfloor-1}{\lfloor{\frac{m}{2}}\rfloor-1}\\
&=&\sum_{i=0}^{\lfloor{\frac{m}{2}}\rfloor-1}z^{m-2i}\binom{m-i}{i}+z^{m-2\lfloor{\frac{m}{2}}\rfloor}\binom{m-\lfloor{\frac{m}{2}}\rfloor-1}{\lfloor{\frac{m}{2}}\rfloor-1}\\
&=&\sum_{i=0}^{\lfloor{\frac{m}{2}}\rfloor}z^{m-2i}\binom{m-i}{i}=B_m.
\end{eqnarray*}
Therefore, the statement holds for both $m$ odd and even.
\end{proof}

\begin{lemma}\label{a}
 Let $A_m(z)= \displaystyle\sum_{i=0}^{\lfloor{\frac{m-1}{2}}\rfloor}z^{m-1-2i}\binom{m-1-i}{i}$, for $m \in \N$. Then  \[A_m=zA_{m-1}+A_{m-2}, \,\,\, \text{for}\, \,m >2. \]
\end{lemma}
\begin{proof} Note that  $A_1 = 1$ and $A_m = B_{m-1}$ for $m>1$, where $B_m$ is defined as in Lemma~\ref{b}. Therefore, the statement follows by substituting $m-1$ for $m$ in Lemma~\ref{b}.
\end{proof}

\begin{lemma}\label{c} Let $C_m(z, a)=\displaystyle \sum_{j=1}^m \displaystyle\sum_{i=0}^{\lfloor{\frac{j-1}{2}}\rfloor}z^{j-2i}a^{j-m}\binom{j-1-i}{i}$ and $B_m(z)=\displaystyle\sum_{i=0}^{\lfloor{\frac{m}{2}}\rfloor}z^{m-2i}\binom{m-i}{i}$, for $m \in \N$. Then the following hold:
\begin{eqnarray*}
C_m &=& a^{-1}C_{m-1}+zB_{m-1}, \, \,\,\text{for} \,\, m >1 \\ \quad C_m &=& C_{m-2}+za^{1-m}+zC_{m-1}, \, \,\,\text{for} \,\, m >2.
\end{eqnarray*}
\end{lemma}

\begin{proof}
Considering $C_m$ and $B_m$ as above, we have the following:
\begin{eqnarray*}
a^{-1}C_{m-1}+zB_{m-1}&=&a^{-1}\sum_{j=1}^{m-1} \sum_{i=0}^{\lfloor{\frac{j-1}{2}}\rfloor}z^{j-2i}a^{j-m+1}\binom{j-1-i}{i}\\&&+z\sum_{i=0}^{\lfloor{\frac{m-1}{2}}\rfloor}z^{m-1-2i}\binom{m-1-i}{i}\\
&=&\sum_{j=1}^m \sum_{i=0}^{\lfloor{\frac{j-1}{2}}\rfloor}z^{j-2i}a^{j-m}\binom{j-1-i}{i}=C_m.
\end{eqnarray*}
Therefore, the first equality holds. We verify now the second equality. We observe first that
\begin{eqnarray*}C_{m-2}+za^{1-m}+zC_{m-1}&=&\sum_{j=1}^{m-2} \sum_{i=0}^{\lfloor{\frac{j-1}{2}}\rfloor}z^{j-2i}a^{j-m+2}\binom{j-1-i}{i}+za^{1-m}\\&&+z\sum_{j=1}^{m-1} \sum_{i=0}^{\lfloor{\frac{j-1}{2}}\rfloor}z^{j-2i}a^{j-m+1}\binom{j-1-i}{i}.
\end{eqnarray*}
In addition, by dividing out by $a^{-m}$, we obtain 
\begin{eqnarray*}
\frac{C_m}{a^{-m}}=za+\sum_{j=2}^m \sum_{i=0}^{\lfloor{\frac{j-1}{2}}\rfloor}z^{j-2i}a^{j}\binom{j-1-i}{i}
\end{eqnarray*} 
and 
\begin{eqnarray*}
\frac{C_{m-2}+za^{1-m}+zC_{m-1}}{a^{-m}}&=&\sum_{j=1}^{m-2} \sum_{i=0}^{\lfloor{\frac{j-1}{2}}\rfloor}z^{j-2i}a^{j+2}\binom{j-1-i}{i}+za\\&&+z\sum_{j=1}^{m-1} \sum_{i=0}^{\lfloor{\frac{j-1}{2}}\rfloor}z^{j-2i}a^{j+1}\binom{j-1-i}{i}.
\end{eqnarray*}
Thus it suffices to show that
\begin{eqnarray} \label{eq:coeffC}
\sum_{j=2}^m \sum_{i=0}^{\lfloor{\frac{j-1}{2}}\rfloor}z^{j-2i}a^{j}\binom{j-1-i}{i}&=&\sum_{j=1}^{m-2} \sum_{i=0}^{\lfloor{\frac{j-1}{2}}\rfloor}z^{j-2i}a^{j+2}\binom{j-1-i}{i}\\
&&+z\sum_{j=1}^{m-1} \sum_{i=0}^{\lfloor{\frac{j-1}{2}}\rfloor}z^{j-2i}a^{j+1}\binom{j-1-i}{i}. \nonumber
\end{eqnarray}
Let's take a look at the right hand side of the above equality.
\begin{eqnarray*}
RHS&=&\sum_{j=1}^{m-2} a^{j+2}\sum_{i=0}^{\lfloor{\frac{j-1}{2}}\rfloor}z^{j-2i}\binom{j-1-i}{i}+z\sum_{j=1}^{m-1} a^{j+1}\sum_{i=0}^{\lfloor{\frac{j-1}{2}}\rfloor}z^{j-2i}\binom{j-1-i}{i}\\
&=&\sum_{j=3}^{m} a^{j}\sum_{i=0}^{\lfloor{\frac{j-3}{2}}\rfloor}z^{j-2-2i}\binom{j-3-i}{i}+z\sum_{j=2}^{m} a^{j}\sum_{i=0}^{\lfloor{\frac{j-2}{2}}\rfloor}z^{j-1-2i}\binom{j-2-i}{i}\\
&=&\sum_{j=3}^{m} a^{j}\sum_{i=0}^{\lfloor{\frac{j-3}{2}}\rfloor}z^{j-2-2i}\binom{j-3-i}{i}+\sum_{j=3}^{m} a^{j}\sum_{i=0}^{\lfloor{\frac{j-2}{2}}\rfloor}z^{j-2i}\binom{j-2-i}{i}+a^2z^2\\
&=&\sum_{j=3}^{m} a^{j}z^j\left[\sum_{i=0}^{\lfloor{\frac{j-3}{2}}\rfloor}z^{-2-2i}\binom{j-3-i}{i}+\sum_{i=0}^{\lfloor{\frac{j-2}{2}}\rfloor}z^{-2i}\binom{j-2-i}{i}\right]+a^2z^2.
\end{eqnarray*}
Let $S$ be the inside of the sum above. Then, we have
\begin{eqnarray*}
S&=&\sum_{i=0}^{\lfloor{\frac{j-1}{2}}\rfloor-1}z^{-2-2i}\binom{j-3-i}{i}+\sum_{i=0}^{\lfloor{\frac{j-2}{2}}\rfloor}z^{-2i}\binom{j-2-i}{i}\\
&=&\sum_{i=1}^{\lfloor{\frac{j-1}{2}}\rfloor}z^{-2i}\binom{j-2-i}{i-1}+\sum_{i=1}^{\lfloor{\frac{j-2}{2}}\rfloor}z^{-2i}\binom{j-2-i}{i}+1.
\end{eqnarray*}
Case 1: If $j$ is even then $\lfloor{\frac{j-1}{2}}\rfloor=\lfloor{\frac{j-2}{2}}\rfloor$. Therefore, we have 
\begin{eqnarray*}
S&=&\sum_{i=1}^{\lfloor{\frac{j-1}{2}}\rfloor}z^{-2i}\left[\binom{j-2-i}{i-1}+\binom{j-2-i}{i}\right]+1\\&=&\sum_{i=1}^{\lfloor{\frac{j-1}{2}}\rfloor}z^{-2i}\binom{j-1-i}{i}+1\\&=&\sum_{i=0}^{\lfloor{\frac{j-1}{2}}\rfloor}z^{-2i}\binom{j-1-i}{i}.
\end{eqnarray*}
Case 2: If $j$ is odd then $\lfloor{\frac{j-1}{2}}\rfloor=\lfloor{\frac{j}{2}}\rfloor$. Then, we have 
\begin{eqnarray*}
S&=&\sum_{i=1}^{\lfloor{\frac{j-1}{2}}\rfloor}z^{-2i}\binom{j-2-i}{i-1}+\sum_{i=1}^{\lfloor{\frac{j}{2}}\rfloor-1}z^{-2i}\binom{j-2-i}{i}+1.
\end{eqnarray*}
Breaking off the last term of the first part, we get
\begin{eqnarray*}
S&=&\sum_{i=1}^{\lfloor{\frac{j-1}{2}}\rfloor-1}z^{-2i}\binom{j-2-i}{i-1}+\sum_{i=1}^{\lfloor{\frac{j}{2}}\rfloor-1}z^{-2i}\binom{j-2-i}{i}+1\\&&+z^{-2\lfloor{\frac{j-1}{2}}\rfloor}\binom{j-2-\lfloor{\frac{j-1}{2}}\rfloor}{\lfloor{\frac{j-1}{2}}\rfloor-1}\\
&=&\sum_{i=1}^{\lfloor{\frac{j-1}{2}}\rfloor-1}z^{-2i}\left[\binom{j-2-i}{i-1}+\binom{j-2-i}{i}\right]+1\\
&&+z^{-2\lfloor{\frac{j-1}{2}}\rfloor}\binom{j-2-\lfloor{\frac{j-1}{2}}\rfloor}{\lfloor{\frac{j-1}{2}}\rfloor-1}\\
&=&\sum_{i=1}^{\lfloor{\frac{j-1}{2}}\rfloor-1}z^{-2i}\binom{j-1-i}{i}+1+z^{-2\lfloor{\frac{j-1}{2}}\rfloor}\binom{j-2-\lfloor{\frac{j-1}{2}}\rfloor}{\lfloor{\frac{j-1}{2}}\rfloor-1}.
\end{eqnarray*}
Since $j$ is odd, $j=2p+1$ for some $p\in \N$. Then $\lfloor{\frac{j-1}{2}}\rfloor=p$, so
\[\binom{j-2-\lfloor{\frac{j-1}{2}}\rfloor}{\lfloor{\frac{j-1}{2}}\rfloor-1}=\binom{j-p-2}{p-1}=\binom{j-p-1}{p}=1,\] for the choice of $j$ above.

Therefore, in either case, we get that the sum $S$ satisfies the following equality: \[S=\sum_{i=0}^{\lfloor{\frac{j-1}{2}}\rfloor}z^{-2i}\binom{j-1-i}{i}.\]

Hence, our full original equation becomes:
\begin{eqnarray*}
RHS&=&\sum_{j=3}^{m} a^{j}z^j\left[\sum_{i=0}^{\lfloor{\frac{j-3}{2}}\rfloor}z^{-2-2i}\binom{j-3-i}{i}+\sum_{i=0}^{\lfloor{\frac{j-2}{2}}\rfloor}z^{-2i}\binom{j-2-i}{i}\right]+a^2z^2\\
&=&\sum_{j=3}^{m} a^{j}z^j\sum_{i=0}^{\lfloor{\frac{j-1}{2}}\rfloor}z^{-2i}\binom{j-1-i}{i}+a^2z^2\\&=&\sum_{j=2}^{m} a^{j}z^j\sum_{i=0}^{\lfloor{\frac{j-1}{2}}\rfloor}z^{-2i}\binom{j-1-i}{i},
\end{eqnarray*} 
which shows that the identity~\eqref{eq:coeffC} holds. Consequently, the second equality in the statement holds.
 \end{proof}

\subsection{The reduction formulas}

Lemmas \ref{b}--\ref{c} and our previous analysis lead to the following result, which we call the \textit{first reduction formula} (compare with~\cite[Theorem 1.14]{P}; we will give more details on this comparison in Remark~\ref{Chebyshev-poly}).

\begin{theorem}\label{red}  
Let $m$ be a positive integer and consider a family of link diagrams $L_0,\, L_-, \, L_{\infty}$ and $L_m$ which are identical except within a ball where they differ as indicated in Figure \ref{fln}. Then
\begin{eqnarray}\label{red-formula}
P(L_m)&= A_mP(L_-) + B_mP(L_0)  - C_mP(L_{\infty}),
\end{eqnarray} where 
\begin{eqnarray*}
A_m&=& B_{m-1} \,\, \text{for}\,\, m>1\,\, \text{and}\, \,A_1 = 1\\
B_m&=&\displaystyle\sum_{i=0}^{\lfloor{\frac{m}{2}}\rfloor}z^{m-2i}\binom{m-i}{i}\\
C_m&=&\displaystyle\sum_{j=1}^m \displaystyle\sum_{i=0}^{\lfloor{\frac{j-1}{2}}\rfloor}z^{j-2i}a^{j-m}\binom{j-1-i}{i}.
\end{eqnarray*}
\end{theorem}
\begin{proof} We will proceed by induction on $m$. If $m=1$, then $A_1=1, B_1=C_1=z$ and the equation~\eqref{red-formula} is merely the Dubrovnik skein relation. Thus the statement holds trivially for $m=1$.

Let $m=2$. Then $A_2=z, B_2=z^2+1,$ and $C_2=za^{-1}+z^2$. By the Dubrovnik skein relation, $P(L_2)=P(L_0)-za^{-1}P(L_{\infty})+zP(L_1)$ and $P(L_1)=P(L_-)-zP(L_{\infty})+zP(L_0)$. Substituting $P(L_1)$ in the identity for $P(L_2)$, we obtain
\begin{eqnarray*}
P(L_2)&=&P(L_0)-za^{-1}P(L_{\infty})+zP(L_-)-z^2P(L_{\infty})+z^2P(L_0)\\
&=&zP(L_-) + (z^2+1)P(L_0) - (za^{-1}+z^2)P(L_{\infty}).
\end{eqnarray*}
Therefore, the equation~\eqref{red-formula} holds for $m=2$.

Suppose that~\eqref{red-formula} holds for all $3 \leq i\leq m-1$, $i\in \N$. In Remark \ref{relation}, we showed that the Dubrovnik skein relation applied at the rightmost crossing in a standard braid-form diagram of odd length can be rewritten as follows: \[P(L_m)= P(L_{m-2})-za^{1-m}P(L_{\infty})+zP(L_{m-1}).\]
Note that by the way we defined $L_m$, the diagrams $L_0, L_{-}$ and $L_{\infty}$ are the same regardless of the chosen $m$. Substituting our assumed formula in for $P(L_{m-1})$ and $P(L_{m-2})$, yields
\begin{eqnarray*}
P(L_m)&=&[A_{m-2}P(L_-) + B_{m-2}P(L_0) - C_{m-2}P(L_{\infty})]-za^{1-m}P(L_{\infty})\\
&&+z[A_{m-1} P(L_-) + B_{m-1}P(L_0) - C_{m-1}P(L_{\infty})]\\
&=&[A_{m-2}+zA_{m-1}]P(L_-) + [B_{m-2}+zB_{m-1}]P(L_0)\\
&&-[C_{m-2}+za^{1-m}+zC_{m-1}]P(L_{\infty}).
\end{eqnarray*}
Employing Lemmas \ref{b}--\ref{c}, we obtain 
\[P(L_m) = A_mP(L_-) + B_mP(L_0) - C_mP(L_{\infty}),\]
which completes the proof.
\end{proof}

\begin{remark}\label{Chebyshev-poly}
The coefficients $A_m(z)$ and $B_m(z)$ used in Theorem~\ref{red} are a version of the Chebyshev polynomials.\ (We thank Sergei Chmutov and Jozef Przytycki for pointing this to us.) The \textit{Chebyshev polynomials of the second kind}, $U_m(x)$, satisfy the initial conditions
\[ U_0(x) =1, \,\, U_1(x) = 2x \]
and the recursive relation
\[U_m (x) = 2x U_{m-1}(x) -U_{m-2}(x), \,\,\, \text{for} \,\, m>1.   \]
It is an easy exercise to verify that
\[ A_m(z) = i^{1-m} U_{m-1}\left (\displaystyle \frac{iz}{2} \right), \,\, \text{for all} \,\,  m \geq 1, \,\,\, \text{where} \,\, i^2 = -1. \]
We remark that Jozef Przytycki obtained (see~\cite{P}) a similar reduction formula  as the one we gave in Theorem~\ref{red} (we thank him for telling us about this). The formula in~\cite[Theorem 1.14]{P} presents the 2-variable Kauffman polynomial (not the Dubrovnik polynomial) of a link diagram $L_m$ with $m>0$ in terms of the polynomials of the associated link diagrams $L_0, L_{\infty}$ and $L_+$ (not $L_-$). The polynomials $v_1^{(m)}(z)$ in~\cite[Theorem 1.14]{P} are related to our coefficient polynomials $A_m(z)$ as follows:
\[A_m(z) = i^{1-m} v_1^{(m)}(iz) \,\, \text{for} \,\, m \geq 1.  \] 
\end{remark}

We consider now the case $m<0$, to obtain the \textit{second reduction formula}. Given a standard diagram $D[b_1, \dots, b_n]$ with $n$ even and $b_i >0$ for all $1 \leq i \leq n$,  we consider a block of $m$ half-twists in the rightmost section of twists in the diagram. Since $n$ is even and all $b_i$'s are positive, the rightmost section of twists in the diagram is in the upper row, which corresponds to the case $m<0$. We give the resulting statement below. 

\begin{theorem}\label{red2}
Let $m$ be a negative integer and consider a family of link diagrams $L_0,\,  L_+,\, L_{\infty}$ and  $L_m$ which are identical except within a ball where they differ as indicated in Figure \ref{fln}.  Then the following equality holds: 
\begin{eqnarray} \label{red2-formula}
P(L_m)&=& A_{m} P(L_+) + B_{m} P(L_0) - C_{m} P(L_{\infty}),
\end{eqnarray} where 
\begin{eqnarray*}
A_{m}&=&B_{m+1} \,\, \text{for} \,\, m< -1 \,\, \text{and} \,\, A_{-1} = 1\\
B_{m}&=&\displaystyle\sum_{i=0}^{\lfloor{\frac{|m|}{2}}\rfloor}(-z)^{|m|-2i}\binom{|m|-i}{i}\\
C_{m} &=&\displaystyle\sum_{j=1}^{|m|} \displaystyle\sum_{i=0}^{\lfloor{\frac{j-1}{2}}\rfloor}(-z)^{j-2i}a^{|m|-j}\binom{j-1-i}{i}.
\end{eqnarray*}
\end{theorem}

\begin{proof}
The proof is similar to the proof for Theorem~\ref{red}; we use induction on  $|m|$. If $m=-1$ then $A_{-1} = 1,\, B_{-1} = C_{-1} = -z$ and if $m = -2$ then $A_{-2} = -z,\, B_{-2} = z^2 +1,  C_{-2} = -za + z^2$.

By the Dubrovnik skein relation, $P(L_{-1}) = P(L_+) -z P(L_0) + z P(L_{\infty})$, and thus the formula~\eqref{red2-formula} holds for $m = -1$. Moreover, $P(L_{-2}) = P(L_0) + za P(L_{\infty}) -z P(L_{-1})$ and equivalently, 
\[P(L_{-2}) = - zP(L_+) + (z^2+1)P(L_0) -(- za +z^2)P(L_{\infty}).\]
Thus the statement holds for $m = -2$. Now we suppose that the statement holds for all negative integers larger than $m$, and we prove it holds for $m$. By the Dubrovnik skein relation, we have
\[P(L_m) = P(L_{m+2}) + za^{-m-1} P(L_{\infty}) - z P(L_{m+1}).   \]
A word of clarification is needed here, since $m$ is negative. In the above notation, the diagrams $L_{m+1}$ and $L_{m+2}$ contain 1 and, respectively, 2 less half-twists is the considered section of twists. 
By the induction hypothesis, 
\begin{eqnarray*}
P(L_{m+1}) &=& A_{m+1} P(L_+) + B_{m+1} P(L_0) - C_{m+1} P(L_{\infty}) \\
P(L_{m+2}) &=&A_{m+2} P(L_+) + B_{m+2} P(L_0) - C_{m+2} P(L_{\infty}).
\end{eqnarray*}
Substituting these into the above expression for $P(L_m)$, we obtain:
\begin{eqnarray*}
 P(L_m) &=& (A_{m+2} -z A_{m+1}) P(L_+) +(B_{m+2} -z B_{m+1})P(L_0) \\
 && - (C_{m+2} -za^{-m -1 } -z C_{m+1}) P(L_{\infty}).
 \end{eqnarray*}
 
 Similar proofs as in Lemmas~\ref{b} -- ~\ref{c} can be given to show that for $m<-2$, we have
 \begin{eqnarray*}
 A_m &=& -z A_{m+1} + A_{m+2}\\
 B_m &=& -z B_{m+1} + B_{m+2}\\
 C_m &=& -z C_{m+1} -z a^{-m +1} + C_{m+2},
 \end{eqnarray*}
 and therefore, the desired formula~\eqref{red2-formula} holds for $m$.
\end{proof}

\begin{remark}\label{sequence-formula}
 Consider a standard braid-form diagram $D[b_1, \ldots, b_n]$ with all $b_i$'s positive (and with $n$ odd or even). Applying Theorems \ref{red} and \ref{red2} recursively for $m=b_n$, we can write the Dubrovnik polynomial $\P[b_1, \ldots, b_n]$ in terms of polynomials associated to standard braid-form diagrams with fewer sections of twists, as follows:
\begin{eqnarray*}
\P[b_1, \ldots, b_n]&=&\sum_{i=0}^{\lfloor{\frac{b_n}{2}}\rfloor}(\epsilon_nz)^{b_n-2i}\binom{b_n-i}{i}\P[b_1, \ldots, b_{n-1}, 0]\\
&& + \sum_{i=0}^{\lfloor{\frac{b_n-1}{2}}\rfloor}(\epsilon_nz)^{b_n-1-2i}\binom{b_n-1-i}{i}\P[b_1, \ldots, b_{n-1},-1]\\
&&-\sum_{j=1}^{b_n} \sum_{i=0}^{\lfloor{\frac{j-1}{2}}\rfloor}(\epsilon_nz)^{j-2i}a^{\epsilon_n(j-b_n)}\binom{j-1-i}{i}\P[b_1, \ldots, b_{n-1}, \infty],
\end{eqnarray*}
where  $\epsilon_n=(-1)^{n-1}$. Using that 
$\P[b_1, \ldots, b_{n-1}, 0]=a^{\epsilon_nb_{n-1}} P[b_1, \ldots, b_{n-2}]$ and that the following diagrams are equivalent as links:
\[
D[b_1, \ldots, b_{n-1},-1] = D [b_1, \ldots, b_{n-1}-1], \,\,\,
D[b_1, \ldots, b_{n-1}, \infty] = D[b_1, \ldots, b_{n-1}], 
\]
we obtain that
\begin{eqnarray}\label{poly-formula}
 \P[b_1, \ldots, b_n]&=&a^{\epsilon_nb_{n-1}}\sum_{i=0}^{\lfloor{\frac{b_n}{2}}\rfloor}(\epsilon_nz)^{b_n-2i}\binom{b_n-i}{i}\P[b_1, \ldots, b_{n-2}] \nonumber \\
&& + \sum_{i=0}^{\lfloor{\frac{b_n-1}{2}}\rfloor}(\epsilon_nz)^{b_n-1-2i}\binom{b_n-1-i}{i}\P[b_1, \ldots, b_{n-1}-1]\\
&&-\sum_{j=1}^{b_n} \sum_{i=0}^{\lfloor{\frac{j-1}{2}}\rfloor}(\epsilon_nz)^{j-2i}a^{\epsilon_n(j-b_n)}\binom{j-1-i}{i}\P[b_1, \ldots, b_{n-1}], \nonumber
\end{eqnarray} 
for all $n >2$.
Note that the sign $\epsilon_n=(-1)^{n-1}$ is necessary for switching in our computations between $m<0$ and $m>0$ (see Figure \ref{fln}), as the calculations switch between reducing lower and upper rows of a standard braid-form diagram. Observe that reducing a lower row (or upper row) corresponds to an odd-length (or even-length) standard braid-form diagram. 
\end{remark}

For a given standard diagram $D[b_1, \ldots, b_{n-1}, b_n]$ with all $b_i$'s positive and $n>2$ (where $n$ is odd or even), let 
\[x_{n, b_n} : = \P[b_1, \ldots,b_{n-1}, b_n]\,\, \text{and} \,\, x_{n, b_n-1}: = \P[b_1, \ldots, b_{n-1}, b_n-1].\]
 More generally, let 
 \[x_{n, k}: = \P[b_1, \ldots,b_{n-1}, k],\,\, \text{where}\,\,  k \leq b_n\] and 
 \[x_{n-1, b_n-1} : = \P[b_1, \ldots,b_{n-1}]\,\, \text{and} \,\, x_{n-2, b_n-2}: = \P[b_1, \ldots,b_{n-2}].\] 
Thus the first subscript $n$ in $x_{n, k}$ stands for the length of a standard diagram and the second subscript $k$ corresponds to the number of half-twists in the rightmost section of twists in the given diagram. It is important to note that the integers $b_1, \ldots, b_{n-1}$ are fixed in the notation for $x_{n, k}$. 

With this new notation, the formula in~\eqref{poly-formula} becomes:
\begin{eqnarray}\label{eq:sequence-formula}
x_{n,b_n}&= &a^{(-1)^{n-1}b_{n-1}}\sum_{i=0}^{\lfloor{\frac{b_n}{2}}\rfloor}((-1)^{n-1}z)^{b_n-2i}\binom{b_n-i}{i}x_{n-2,b_{n-2}} \nonumber \\
&&+ \sum_{i=0}^{\lfloor{\frac{b_n-1}{2}}\rfloor}((-1)^{n-1}z)^{b_n-1-2i}\binom{b_n-1-i}{i}x_{n-1,b_{n-1}-1}\\&&-\sum_{j=1}^{b_n} \sum_{i=0}^{\lfloor{\frac{j-1}{2}}\rfloor}((-1)^{n-1}z)^{j-2i}a^{(-1)^{n-1}(j-b_n)}\binom{j-1-i}{i}x_{n-1, b_{n-1}}, \nonumber
\end{eqnarray}
This formula gives the Dubrovnik polynomial of a standard braid-form diagram of length $n$ in terms of the polynomials of standard diagrams of lengths $n-1$ and $n-2$, greatly reducing the amount of the necessary iterations, when compared to the defining skein relation alone. 

For a standard diagram $D[b_1, \ldots, b_{n-1}, b_n]$ with $b_i$'s positive and $n>2$, we let
\begin{eqnarray}
r_{n,b_n}: &=&a^{(-1)^{n-1}b_{n-1}}\sum_{i=0}^{\lfloor{\frac{b_n}{2}}\rfloor}((-1)^{n-1}z)^{b_n-2i}\binom{b_n-i}{i} \label{eq:r}\\
p_{n,b_n}: &=&\sum_{i=0}^{\lfloor{\frac{b_n-1}{2}}\rfloor}((-1)^{n-1}z)^{b_n-1-2i}\binom{b_n-1-i}{i} \label{eq:p}\\
l_{n,b_n}: &=&-\sum_{j=1}^{b_n} \sum_{i=0}^{\lfloor{\frac{j-1}{2}}\rfloor}((-1)^{n-1}z)^{j-2i}a^{(-1)^{n-1}(j-b_n)}\binom{j-1-i}{i}. \label{eq:l}
\end{eqnarray}
The identity~\eqref{eq:sequence-formula} becomes
\begin{eqnarray*}
x_{n,b_n} = l_{n,b_n}x_{n-1,b_{n-1}} + r_{n,b_n}x_{n-2,b_{n-2}} + p_{n,b_n}x_{n-1,b_{n-1}-1}
\end{eqnarray*}
or equivalently, 
\begin{eqnarray} \label{eq:sequences}
&& \P[b_1, \ldots, b_n] \\
 && =  l_{n,b_n} \P[b_1, \ldots, b_{n-1}] + r_{n,b_n} \P[b_1, \ldots, b_{n-2}] + p_{n,b_n}\P[b_1, \ldots, b_{n-1}-1].     \nonumber
\end{eqnarray}

Note that the terms in equation \eqref{eq:sequences} depend on the fixed integers $b_{n-1}$ and $b_{n-2}$, which are determined by the form of the original standard diagram $D[b_1, \dots, b_{n-1}, b_n]$.

Applying the Dubrovnik skein relation for diagrams $D[b_1]$ and $D[b_1, b_2]$ with both $b_1$ and $b_2$ positive, we see that the identities~\eqref{eq:sequence-formula} and~\eqref{eq:sequences} hold for $n=1$ and $n=2$ if we set 
\[x_{0,b_0}=1,\,\,\,\,  x_{-1,b_{-1}} = z^{-1}a + 1 - z^{-1} a^{-1} \,\,  \text{and}\,\,  x_{0,b_0-1} = a^{-1}.\]
We define $b_0=0$ and $b_{-1}=0$, and therefore
\[x_{0,b_0}=x_{0,0}=1\]
\[ x_{-1,b_{-1}} =x_{-1,0}= z^{-1}a + 1 - z^{-1} a^{-1}\]  \[x_{0,b_0-1} =x_{0,-1}= a^{-1}.\]
Note that $x_{1, 0} = P\left(\raisebox{-6pt}{\includegraphics[scale=.2]{unknot}}\right) = 1$. 

Therefore, with the above conventions, the equalities~\eqref{eq:sequence-formula} and~\eqref{eq:sequences} hold for all $n \in \N$.

\section{A closed-form formula for the Dubrovnik polynomial}\label{closed-form}

Now that we have a formula for the Dubrovnik polynomial of a standard braid-form diagram in terms of polynomials of standard diagrams of shorter lengths, it is one more step to obtain a closed-form expression for the Dubrovnik polynomial of the original diagram. Specifically, we seek a closed-form expression for $x_{n,b_n}$ in terms of $a, z$ and the entries of the $n$-tuple $(b_1, \ldots, b_{n-1}, b_n)$ associated with a standard diagram $D[b_1, \ldots, b_{n-1}, b_n]$. For this, it is convenient to consider first a more general situation represented by the following lemma.

\begin{lemma}\label{sequence} 
 Given a standard braid-form diagram $D[b_1, \ldots, b_{n-1}, b_n]$ of length $n$ and with $b_i >0$ for all $1\leq i\leq n$, let $l_{i, j}, \, r_{i, j}$ and $p_{i, j}$ be elements of a certain commutative ring $R$ (here we are interested in $R = \Z[a^{\pm 1}, z^{\pm 1}]$) and define recurrently the sequence $x_{n,b_n}, n\geq -1$, of elements in $R$ by the relation
\begin{eqnarray}\label{recurrence}
\hspace{1cm} x_{n,b_n}=l_{n,b_n}x_{n-1,b_{n-1}}+r_{n,b_n}x_{n-2,b_{n-2}}+p_{n,b_n}x_{n-1,b_{n-1}-1}, \,\,\,\text{for} \,\, n \geq 1
\end{eqnarray}
where $x_{-1,0}, \, x_{0,0},$ and $x_{0,-1}$ are fixed elements in $R$, and $b_0=b_{-1}=0$. 
Let $F$ be the set of all strictly decreasing integer sequences $f=\{f_1, f_2, \ldots, f_l\}$, where $f_1=n,\,\, f_i-f_{i+1}=1$ or $2$,\,\, $f_l=0$ or $-1$, and only one of $0$ or $-1$ is present in a sequence $f$ (that is, if $f_l=-1$ then $f_{l-1}\neq 0$). Each $f_i=c_i$ or $d_i$, and if $f_i = d_i$ then $f_i -f_{i+1} = 1$. 
For each $f_i \in f$, let $t_{f_i} = \begin{cases} 1 \,\,\, \text{if} \,\,\, f_{i-1} = d_{i-1}\\
 0 \,\,\, \text{if} \,\,\, f_{i-1} = c_{i-1}.
  \end{cases}$
  
Then $x_{n,b_n}$ can be written in terms of the initial conditions $x_{-1,0},\,  x_{0,0}, $ and $x_{0,-1}$ and the elements $l_{i,j}, \, r_{i,j}$ and $p_{i,j}$  as follows:
\begin{eqnarray}\label{tree-recurrence}
\hspace{1cm}x_{n,b_n}=\sum_{f\in F}x_{f_l,b_{f_l}-t_{f_l}}\prod_{i\in \lambda(f)} l_{f_i,b_{f_i}-t_{f_i}} \prod_{i\in\rho(f)} r_{f_i, b_{f_i}-t_{f_i}} \prod_{i\in \gamma(f)} p_{f_i, b_{f_i}-t_{f_i}}
\end{eqnarray}
where 
\begin{eqnarray*}
\lambda(f) &=&\{i| f_i = c_i\,\, \text{and} \,\, c_i-f_{i+1}=1\}\\
\rho(f) &=& \{i|f_i = c_i\,\, \text{and} \,\,c_i-f_{i+1}=2\}\\
\gamma(f) &=&  \{i|f_i = d_i \,\, \text{and} \,\,d_i-f_{i+1}=1\}.
\end{eqnarray*}
 \end{lemma}

\begin{proof} The formula describes the tree of calculations resulting from successive applications of the recurrence formula~\eqref{recurrence}.
We represent the computational tree using a layered rooted tree (as in~\cite{D}), where each layer $i$ collects the coefficients $l_{i, j}, r_{i, j}$, and $p_{i, j}$. Moreover, each layer $i$ corresponds to all terms $x_{i, j}$ whose first subscript is $i$. The $l$-edges and $r$-edges are drawn on the left and middle, respectively, at a vertex in the tree. According to the formula for $x_{n,b_n}$, the $l$-edges are of length 1 and $r$-edges of length 2; that is, the $l$-edges connect vertices in the tree that are at one level apart, and the $r$-edges connect vertices that are at two levels apart. Similarly, the $p$-edges are of length 1. We draw the $p$-edges to the right, and color them red. Figure \ref{3poly} shows a computational tree with $n=3$.
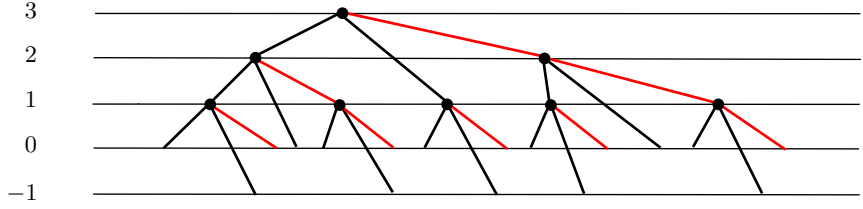
\begin{figure}[ht]
\raisebox{-10pt}{\includegraphics[scale=.6]{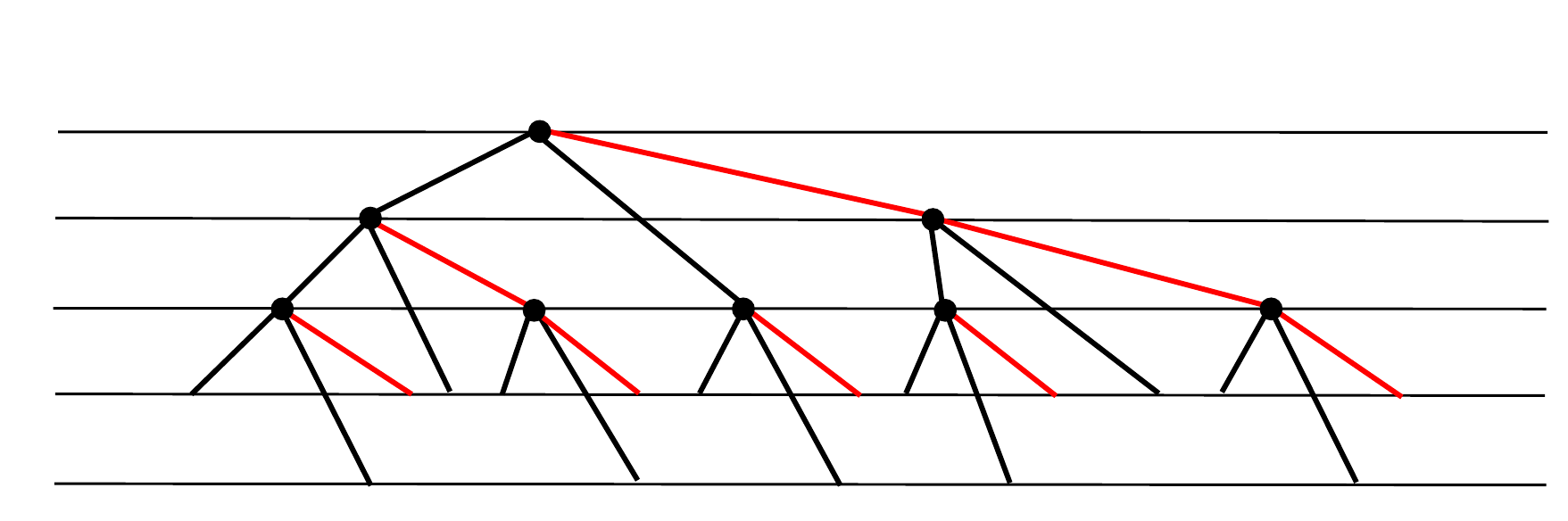}}
 \put(-320,64){\fontsize{9}{10}$ 3$}
  \put(-320, 47){\fontsize{9}{10}$ 2$}
    \put(-320, 30){\fontsize{9}{10}$ 1$}
  \put(-320, 13){\fontsize{9}{10}$ 0$}
    \put(-327, -5){\fontsize{9}{10}$ -1$}
\caption{Computational tree for $n=3$}\label{3poly}
 \end{figure}
The set of integer sequences $F$ generates the possible paths from the root to a leaf of the tree. Each sequence $f$ corresponds to a unique branch in the tree, with $f_i$ representing a vertex of the path.  An entry $c_i$ in a sequence $f$ corresponds to traveling along the left or middle edge incident to a vertex $f_i$ located at the $i$-th level in the tree. On the other hand, an entry $d_i$ in $f$ corresponds to traveling along the right edge at that vertex.  Any path in the tree produces a product of $l$-, $r$-, and $p$-coefficients, and the expression for $x_{n,b_n}$ is the sum over all paths in the tree, where each path starts from the highest level and ends at a leaf of the tree.  
\end{proof}

To obtain our desired closed-form formula for the Dubrovnik polynomial of a standard braid-form diagram $D[b_1, \ldots, b_n]$ with all $b_i >0$ positive, we combine Lemma \ref{sequence} with the reduction formulas provided in Theorem~\ref{red} and Theorem~\ref{red2} (and the associated formula~\eqref{eq:sequences} in Remark~\ref{sequence-formula}). Specifically, we associate the vertices in the layered tree of Lemma \ref{sequence} with the (Dubrovnik polynomial of) rational knot diagrams in standard braid-form obtained by successive applications of the reduction formulas in Theorem \ref{red} and Theorem~\ref{red2}. The root of the tree corresponds to the given rational knot diagram. Each vertex $v$ in the tree is associated with a diagram $L_m$ as in Theorem \ref{red} (if $m>0$) or Theorem~\ref{red2} (if $m<0$), and  the three descendants of $v$ correspond to the diagrams $L_{\infty}, L_0$ and $L_-$ (or $L_+$) associated with $L_m$. Figure \ref{3poly} depicts the computational tree for a standard braid-form diagram of length 3.

Notice that the $l$-, $r$-, and $p$-coefficients correspond to the coefficient polynomials $C_m(z, a), B_m(z)$ and $A_m(z)$, respectively, introduced in Section~\ref{sec:reduction}. The $p$-edges in a computational tree are highlighted in red to differentiate them from the other edges in the tree, as they reduce a different standard diagram than the other edges, namely $D[b_1,\ldots, b_i-1]$, decreasing not only the number of sections of twists in the diagram but also the number of half-twists in the rightmost section of twists.  

\begin{theorem}\label{closedform}
Let $D[b_1, \ldots, b_n]$ be a standard braid-form diagram with $b_i > 0$ for all $1 \leq i \leq n$. Then
\begin{eqnarray*}
&&\P[b_1, \ldots, b_n] =\\
&&\sum_{f\in F}x_{f_l,b_{f_l}-t_{f_l}} \prod_{i\in \lambda(f)}-\sum_{j=1}^{b_{f_i}-t_{f_i}} \sum_{i=0}^{\lfloor{\frac{j-1}{2}}\rfloor}((-1)^{f_i-1}z)^{j-2i}a^{(-1)^{f_i-1}(j-b_{f_i}+t_{f_i})}\binom{j-1-i}{i}\\
&& \cdot \prod_{i\in\rho(f)}  a^{(-1)^{f_i-1}b_{f_i-1}}\sum_{i=0}^{\lfloor{\frac{b_{f_i}-t_{f_i}}{2}}\rfloor}((-1)^{f_i-1}z)^{b_{f_i}-t_{f_i}-2i}\binom{b_{f_i}-t_{f_i}-i}{i} \\
&& \cdot \prod_{i\in \gamma(f)} \sum_{i=0}^{\lfloor{\frac{b_{f_i}-t_{f_i}-1}{2}}\rfloor}((-1)^{f_i-1}z)^{b_{f_i}-t_{f_i}-1-2i}\binom{b_{f_i}-t_{f_i}-1-i}{i},
\end{eqnarray*}
where 
\begin{itemize}
\item $F$ is the set of all strictly decreasing integer sequences $f=\{f_1, f_2, \ldots, f_l\}$, where $f_1=n, f_i-f_{i+1}=1$ or $2$, $f_l=0$ or $-1$, and only one of $0$ or $-1$ is present in the sequence $f$ (that is, if $f_l=-1$ then $f_{l-1}\neq 0$). Each $f_i=c_i$ or $d_i$, and if $f_i = d_i$ then $f_{i} -f_{i+1} = 1$. For $f_i \in f$, $t_{f_i} = \begin{cases} 1 \,\,\, \text{if} \,\,\, f_{i-1} = d_{i-1}\\
 0 \,\,\, \text{if} \,\,\, f_{i-1} = c_{i-1}.
  \end{cases}$
\item $\lambda(f)=\{i|f_i = c_i \,\, \text{and} \,\,c_i-f_{i+1}=1\}$,
\item $\rho(f)=\{i| f_i = c_i \,\,\text{and} \,\, c_i-f_{i+1}=2\}$,
\item $\gamma(f)=\{i| f_i = d_i \,\, \text{and}\,\, d_i-f_{i+1}=1\}$,
\item $x_{0,0}=1,\,\, x_{-1,0} = z^{-1}a + 1 - z^{-1} a^{-1}$ and $x_{0,-1} = a^{-1}$.
\end{itemize}
\end{theorem}

\begin{proof} The statement follows from Lemma \ref{sequence} and Remark \ref{sequence-formula}. \end{proof}

The Dubrovnik polynomial of a standard braid-form diagram $D[b_1, \ldots, b_n]$ with all $b_i <0$ is obtained from the closed-form formula (given in Theorem~\ref{closedform}) for the Dubrovnik polynomial for the diagram $D[|b_1|, \ldots, |b_n|]$, in which one is applying the replacements $a \longleftrightarrow a^{-1}$ and $z \longleftrightarrow -z$. That is,
\[ \P[b_1, \ldots, b_n] (a, z) = \P[|b_1|, \ldots, |b_n|] (a^{-1}, -z), \,\,\, b_i <0,  \,\, \forall  1 \leq i \leq n.\]

\begin{example} We compute the Dubrovnik polynomial for the standard diagram $D[4,3,5]$ depicted in Figure \ref{435}.
\begin{figure}[ht]
\[\includegraphics[scale=.25]{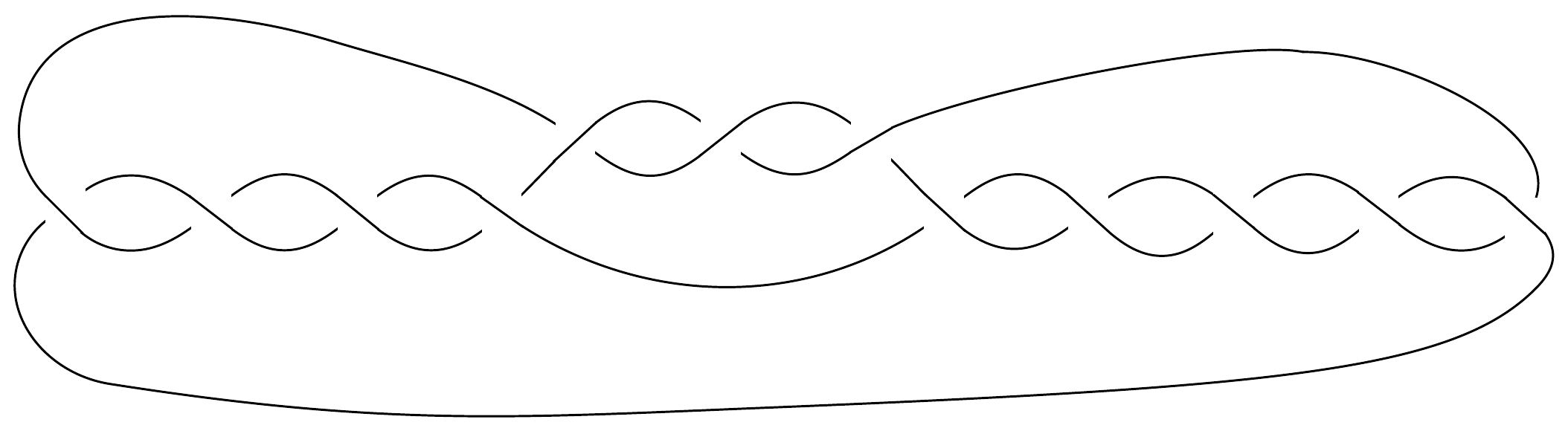}\]
\caption{Standard diagram $D[4,3,5]$ \label{435}}
 \end{figure}

The length of the diagram is $n=3$ and there are $17$ possibilities for the sequences in $F$, which are given in the first column of Table~\ref{table:example}. To differentiate $c_i$'s and $d_i$'s in a sequences $f \in F$, we denote $d_i$'s with a subscript $p$.  In this table, the second and third columns list the terms $c_i \in f$ corresponding to $i \in \lambda(f)$ and $i \in \rho(f)$, respectively. The fourth column lists the terms $d_i \in f$ corresponding to $i \in \gamma(f)$. The last column collects the product of the coefficients corresponding to each branch in the tree; each such product is a term in the Dubrovnik polynomial $\P[4,3,5]$ associated with the standard diagram $D[4,3,5]$.
 
 \begin{table}[ht]
    \begin{tabular}{|c|c|c|c|c|}
\hline
    $f$   & $\lambda(f)\rightarrow \{c_i\}$   & $\rho(f)\rightarrow \{c_i\}$   & $\gamma(f)\rightarrow \{d_i\}$   & products  of coefficients  \\\hline
  \{3,2,1,0\}         & $\{3,2,1\}$ & $\emptyset$       & $\emptyset$                 & $l_{3,5}l_{2, 3}l_{1,4}x_{0,0}$            \\\hline 
 $\{3,2,1,-1\}$      & $\{3,2\}$     & $\{1\}$     & $\emptyset$                 & $l_{3,5}l_{2,3}r_{1,4}x_{-1,0}$         \\\hline 
 $\{3,2,1_p,0\}$     & $\{3,2\}$     & $\emptyset$       & $\{1_p\}$             & $l_{3,5}l_{2,3}p_{1,4}x_{0,-1}$            \\\hline 
 $\{3,2,0\}$         &$ \{3\}$         & $\{2\}$     & $\emptyset$                 & $l_{3,5}r_{2,3}x_{0,0}$               \\\hline 
 $\{3,2_p,1,0\}$     & $\{3,1\}$     & $\emptyset$       & $\{2_p\}$             & $l_{3,5}p_{2,3}l_{1,3}x_{0,0}$        \\\hline 
 $\{3,2_p,1,-1\}$    & $\{3\}$         & $\{1\}$     & $\{2_p\}$             & $l_{3,5}p_{2,3}r_{1,3}x_{-1,0}$     \\\hline 
 $\{3,2_p,1_p,0\}$   & $\{3\}$         & $\emptyset$       & $\{2_p, 1_p\}$      & $l_{3,5}p_{2,3}p_{1,3}x_{0,-1}$        \\\hline 
 $\{3,1,0\}$         & $\{1\}$         & $\{3\}$     & $\emptyset$                 & $r_{3,5}l_{1,4}x_{0,0}$               \\\hline 
 $\{3,1,-1\}$        & $\emptyset$           & $\{3,1\}$ & $\emptyset$                 & $r_{3,5}r_{1,4}x_{-1,0}$            \\\hline 
 $\{3,1_p,0\}$       & $\emptyset$           & $\{3\}$     & $\{1_p\}$             & $r_{3,5}p_{1,4}x_{0,-1}$               \\\hline 
 $\{3_p,2,1,0\}$     & $\{2,1\}$     & $\emptyset$       & $\{3_p\}$             & $p_{3,5}l_{2,2}l_{1,4}x_{0,0}$        \\\hline 
 $\{3_p,2,1,-1\}$    & $\{2\}$         & $\{1\}$     & $\{3_p\}$             & $p_{3,5}l_{2,2}r_{1,4}x_{-1,0}$     \\\hline 
 $\{3_p,2,1_p,0\}$   & $\{2\}$         & $\emptyset$       & $\{3_p,1_p\}$       & $p_{3,5}l_{2,2}p_{1,4}x_{0,-1}$        \\\hline 
 $\{3_p,2,0\}$       & $\emptyset$           & $\{2\}$     & $\{3_p\}$             & $p_{3,5}r_{2,2}x_{0,0}$           \\\hline 
 $\{3_p,2_p,1,0\}$   & $\{1\}$       & $\emptyset$       & $\{3_p,2_p\}$       & $p_{3,5}p_{2,2}l_{1,3}x_{0,0}$    \\\hline 
 $\{3_p,2_p,1,-1\}$  & $\emptyset$           & $\{1\}$     & $\{3_p,2_p\}$       & $p_{3,5}p_{2,2}r_{1,3}x_{-1,0}$ \\\hline 
 $\{3_p,2_p,1_p,0\}$ & $\emptyset$           & $\emptyset$       & $\{3_p,2_p,1_p\}$ & $p_{3,5}p_{2,2}p_{1,3}x_{0,-1}$    \\\hline 
    \end{tabular}
    \vspace{0.2cm}
    \caption{Paths and terms for computing $\P[4,3,5]$} \label{table:example}
\end{table}
The tree that corresponds to a standard diagram of length $3$ is given in Figure~\ref{3poly}. Recall that each sequence $f$ corresponds to a path in the tree starting from the root at level $3$ and ending at a leaf, recording the level for each vertex in the path and writing a subscript $p$ (which marks $d_i$'s) for the vertices where we chose the rightmost edge highlighted in red. 

To see how and why the computational tree works, let the vertices of the tree correspond to the following polynomials, which we evaluate using the formula~\eqref{recurrence} in Lemma \ref{sequence}: 
\begin{eqnarray*}
\P[4,3,5]&=&l_{3,5}\P[4,3]+r_{3,5}\P[4]+p_{3,5}\P[4,2]\\
\P[4,3]&=&l_{2,3}\P[4]+r_{2,3}x_{0,0}+p_{2,3}\P[3]\\
\P[4,2]&=&l_{2,2}\P[4]+r_{2,2}x_{0,0}+p_{2,2}\P[3]\\
\P[4]&=&l_{1,4}x_{0,0}+r_{1,4}x_{-1,0}+p_{1,4}x_{0,-1}\\
\P[3]&=&l_{1,3}x_{0,0}+r_{1,3}x_{-1,0}+p_{1,3}x_{0,-1}.
\end{eqnarray*}
Thus 
\begin{eqnarray*}
\P[4,3]&=&l_{2,3}\P[4]+r_{2,3}x_{0,0}+p_{2,3}\P[3]\\
&=&l_{2,3}(l_{1,4}x_{0,0}+r_{1,4}x_{-1,0}+p_{1,4}x_{0,-1})+r_{2,3}x_{0,0}\\&&+p_{2,3}(l_{1,3}x_{0,0}+r_{1,3}x_{-1,0}+p_{1,3}x_{0,-1})\\
&=&l_{2,3}l_{1,4}x_{0,0}+l_{2,3}r_{1,4}x_{-1,0}+l_{2,3}p_{1,4}x_{0,-1}+r_{2,3}x_{0,0}\\&&+p_{2,3}l_{1,3}x_{0,0}+p_{2,3}r_{1,3}x_{-1,0}+p_{2,3}p_{1,3}x_{0,-1}
\end{eqnarray*}
and
\begin{eqnarray*}
\P[4,2]&=&l_{2,2}\P[4]+r_{2,2}x_{0,0}+p_{2,2}\P[3]\\
&=&l_{2,2}(l_{1,4}x_{0,0}+r_{1,4}x_{-1,0}+p_{1,4}x_{0,-1})+r_{2,2}x_{0,0}\\&&+p_{2,2}(l_{1,3}x_{0,0}+r_{1,3}x_{-1,0}+p_{1,3}x_{0,-1})\\
&=&l_{2,2}l_{1,4}x_{0,0}+l_{2,2}r_{1,4}x_{-1,0}+l_{2,2}p_{1,4}x_{0,-1}+r_{2,2}x_{0,0}\\&&+p_{2,2}l_{1,3}x_{0,0}+p_{2,2}r_{1,3}x_{-1,0}+p_{2,2}p_{1,3}x_{0,-1}.
\end{eqnarray*}
Therefore, we have 
\begin{eqnarray*}
\P[4,3,5]&=&l_{3,5}\P[4,3]+r_{3,5}P(D[4])+p_{3,5}\P[4,2]\\
&=&l_{3,5}l_{2,3}l_{1,4}x_{0,0}+l_{3,5}l_{2,3}r_{1,4}x_{-1,0}+l_{3,5}l_{2,3}p_{1,4}x_{0,-1}\\&&+l_{3,5}r_{2,3}x_{0,0}+l_{3,5}p_{2,3}l_{1,3}x_{0,0}+l_{3,5}p_{2,3}r_{1,3}x_{-1,0}
\\&&+l_{3,5}p_{2,3}p_{1,3}x_{0,-1}+r_{3,5}l_{1,4}x_{0,0}+r_{3,5}r_{1,4}x_{-1,0}\\&&+r_{3,5}p_{1,4}x_{0,-1}+p_{3,5}l_{2,2}l_{1,4}x_{0,0}+p_{3,5}l_{2,2}r_{1,4}x_{-1,0}
\\&&+p_{3,5}l_{2,2}p_{1,4}x_{0,-1}+p_{3,5}r_{2,2}x_{0,0}+p_{3,5}p_{2,2}l_{1,3}x_{0,0}\\&&+p_{3,5}p_{2,2}r_{1,3}x_{-1,0}+p_{3,5}p_{2,2}p_{1,3}x_{0,-1}.
\end{eqnarray*}
Note that this result coincides with the sum of the products in the last column in Table~\ref{table:example}, which are obtained from the computational tree in Figure~\ref{3poly} and its associated paths.
Using the relations~\eqref{eq:r} through~\eqref{eq:l} and the chart of paths, we obtain  
\begin{eqnarray*}
&&\P[4,3,5]=\\
&&-1+\frac{1}{a^4}-2 a^2+3 a^4+\frac{2 z}{a^5}+\frac{z}{a}-5 a z+2 a^3 z-18 z^2+\frac{2 z^2}{a^8}+\frac{3 z^2}{a^6}+\frac{6 z^2}{a^4}\\&&-6 a^2 z^2+13 a^4 z^2+\frac{3 z^3}{a^7}+\frac{12 z^3}{a^5}+\frac{5 z^3}{a^3}
+\frac{5 z^3}{a}-26 a z^3+a^3 z^3-43 z^4+\frac{z^4}{a^8}\\&&+\frac{5 z^4}{a^6}+\frac{17 z^4}{a^4}+\frac{10 z^4}{a^2}-6 a^2 z^4+16 a^4 z^4+\frac{2 z^5}{a^7}+\frac{10 z^5}{a^5}+\frac{12 z^5}{a^3}+\frac{15 z^5}{a}\\&&-33 a z^5-6 a^3 z^5-36z^6+\frac{3 z^6}{a^6}+\frac{12 z^6}{a^4}+\frac{12 z^6}{a^2}+2 a^2 z^6+7 a^4 z^6+\frac{4 z^7}{a^5}\\
&&+\frac{9 z^7}{a^3}+\frac{10 z^7}{a}-18 a z^7-5 a^3 z^7-16 z^8+\frac{4 z^8}{a^4}+\frac{7 z^8}{a^2}+4 a^2 z^8+a^4 z^8\\
&&+\frac{3 z^9}{a^3}+\frac{4 z^9}{a}-6 a z^9-a^3 z^9-3 z^{10}+\frac{2 z^{10}}{a^2}+a^2 z^{10}+\frac{z^{11}}{a}-a z^{11},
\end{eqnarray*} where we used that
\begin{eqnarray*}
x_{0,0}&=&1,\,\,\, x_{-1,0} =z^{-1}a + 1 - z^{-1} a^{-1},  \,\,\, x_{0,-1} = a^{-1}\\
l_{3,5}&=& -1 (z a^{-4} + z^2 a^{-3} + z^3 a^{-2}+ z a^{-2}+ z^4 a^{-1}+ 2 z^2 a^{-1} + z^5 + 3 z^3 + z)\\
r_{3,5} &= &a^3 (z^5 + 4 z^3 + 3 z), \,\,\, p_{3,5}=1 + 3 z^2 + z^4\\
l_{2,3} &=& -1 (-z - z a^2 + z^2 a - z^3), \,\,\, r_{2,3}=-1 (2 z + z^3) a^{-4}, \,\,\, p_{2,3}= 1 + z^2\\
l_{1,4}&=& -1 (z a^{-3} + z^2 a^{-2} + z^3 a^{-1}+ z a^{-1}+ z^4 + 2 z^2)\\
r_{1,4}&=&z^4 + 3 z^2 + 1, \,\,\, p_{1,4} =z^3 + 2 z\\
l_{1,3}&=& -(z + z a^{-2} + z^2 a^{-1} + z^3),\,\,\,  r_{1,3}= 2 z + z^3, \,\,\, p_{1,3}=1 + z^2\\
l_{2,2}&=& z a - z^2, \,\,\, r_{2,2} = a^{-4} (z^2 + 1), \,\,\, p_{2,2} = -z.
\end{eqnarray*}
\end{example}


\section{Appendix} \label{app}

This appendix contains the Mathematica\textsuperscript{\textregistered} code written by the second named author, which computes the Dubrovnik polynomial $\P[b_1, b_2, \dots, b_n]$ based on successive iterations of the Dubrovnik skein relation applied to the left hand side of a standard braid-form diagram of a rational knot, as shown in Section~\ref{sec:Dub}.
\begin{verbatim}
Clear[f, a, c, d, b1, z, b2, b3, b4, b5, btail, i]
f[1, {1}] = c;
f[1, {2}] = d;
f[1, {b1_}] := 
 f[1, {b1}] = f[1, {b1 - 2 }] - z (a^(1 - b1) - f[1, {b1 - 1}])

f[2, {1, 1}] := d;
f[2, {1, 2}] := f[2, {1, 2}] = a^(-1) - z (d - a^2);
f[2, {2, 1}] := f[2, {2, 1}] = a - z a^(-2) + z (f[2, {1, 1}]);
f[2, {1, b2_}] := 
  f[2, {1, b2}] = f[2, {1, b2 - 2}] - z (f[2, {1, b2 - 1}] - a^b2);
f[2, {2, b2_}] :=  f[2, {2, b2}] = 
  a^b2 - z (a^(-1) f[2, {1, b2 - 1}] - f[2, {1, b2}]);
f[2, {b1_, 1}] := f[2, {b1, 1}] = f[1, {b1 + 1}];
f[2, {b1_, b2_}] :=  f[2, {b1, b2}] = 
  f[2, {b1 - 2, b2}] -  z (a^(1 - b1) f[2, {1, b2 - 1}] 
   - f[2, {b1 - 1, b2}])

f[3, {1, 1, b3_}] :=  f[3, {1, 1, b3}] = 
  a^(-b3) - z (f[1, {b3 + 1}] - a f[1, {b3}]);
f[3, {1, 2, b3_}] := f[3, {1, 2, b3}] = 
   f[1, {b3 + 1}] - z (f[3, {1, 1, b3}] - a^2 f[1, {b3}]);
f[3, {2, 1, b3_}] :=  f[3, {2, 1, b3}] = 
   a f[1, {b3}] - z (a^(-1) f[1, {b3 + 1}] - f[3, {1, 1, b3}]);
f[3, {1, b2_, b3_}] :=  f[3, {1, b2, b3}] = 
   f[3, {1, b2 - 2, b3}] - z (f[3, {1, b2 - 1, b3}] 
    - a^b2 f[1, {b3}]);
f[3, {2, b2_, b3_}] :=  f[3, {2, b2, b3}] = 
   a^b2 f[1, {b3}] - z (a^(-1) f[3, {1, b2 - 1, b3}] 
    - f[3, {1, b2, b3}]);
f[3, {b1_, 1, b3_}] :=  f[3, {b1, 1, b3}] = 
   f[3, {b1 - 2, 1, b3}] - z (a^(1 - b1) f[1, {b3 + 1}] 
     - f[3, {b1 - 1, 1, b3}]);
f[3, {b1_, b2_, b3_}] :=  f[3, {b1, b2, b3}] = 
  f[3, {b1 - 2, b2, b3}] - z (a^(-b1 + 1) f[3, {1, b2 - 1, b3}] 
   - f[3, {b1 - 1, b2, b3}])

f[4, {1, 1, b3_, 1}] :=  f[4, {1, 1, b3, 1}] = 
   a^(-1 - b3) - z (f[2, {b3 + 1, 1}] - a f[2, {b3, 1}]);
f[4, {1, 1, b3_, b4_}] :=  f[4, {1, 1, b3, b4}] = 
   a^(-b3) f[2, {1, b4 - 1}] - z (f[2, {b3 + 1, b4}] 
    - a f[2, {b3, b4}]);
f[4, {2, 1, b3_, b4_}] :=  f[4, {2, 1, b3, b4}] = 
   a f[2, {b3, b4}] - z (a^(-1) f[2, {b3 + 1, b4}] 
    - f[4, {1, 1, b3, b4}]);
f[4, {1, 2, b3_, b4_}] := f[4, {1, 2, b3, b4}] = 
   f[2, {b3 + 1, b4}] - z (f[4, {1, 1, b3, b4}] 
    - a^2 f[2, {b3, b4}]);
f[4, {1, b2_, b3_, b4_}] :=  f[4, {1, b2, b3, b4}] = 
   f[4, {1, b2 - 2, b3, b4}] -  z (f[4, {1, b2 - 1, b3, b4}] 
    - a^b2 f[2, {b3, b4}]);
f[4, {2, b2_, b3_, b4_}] :=  f[4, {2, b2, b3, b4}] = 
   a^b2 f[2, {b3, b4}] - z (a^(-1) f[4, {1, b2 - 1, b3, b4}] 
    - f[4, {1, b2, b3, b4}]);
f[4, {b1_, 1, b3_, b4_}] :=  f[4, {b1, 1, b3, b4}] = 
   f[4, {b1 - 2, 1, b3, b4}] - z (a^(1 - b1) f[2, {b3 + 1, b4}] 
    - f[4, {b1 - 1, 1, b3, b4}]);
f[4, {b1_, b2_, b3_, b4_}] :=  f[4, {b1, b2, b3, b4}] = 
  f[4, {b1 - 2, b2, b3, b4}] 
   - z (a^(-b1 + 1) f[4, {1, b2 - 1, b3, b4}] 
     - f[4, {b1 - 1, b2, b3, b4}])

f[5, {1, 1, b3_, 1, b5_}] :=  f[5, {1, 1, b3, 1, b5}] = 
   a^(-b3) f[1, {1 + b5}] - z (f[3, {b3 + 1, 1, b5}] 
    - a f[3, {b3, 1, b5}]);
f[5, {1, 1, b3_, b4_, b5_}] := f[5, {1, 1, b3, b4, b5}] = 
   a^(-b3) f[3, {1, b4 - 1, b5}] - z (f[3, {b3 + 1, b4, b5}] 
    - a f[3, {b3, b4, b5}]);
f[5, {2, 1, b3_, b4_, b5_}] :=  f[5, {2, 1, b3, b4, b5}] = 
   a f[3, {b3, b4, b5}] - z (a^(-1) f[3, {b3 + 1, b4, b5}] 
    - f[5, {1, 1, b3, b4, b5}]);
f[5, {1, 2, b3_, b4_, b5_}] := f[5, {1, 2, b3, b4, b5}] = 
   f[3, {b3 + 1, b4, b5}] - z (f[5, {1, 1, b3, b4, b5}] 
    - a^2 f[3, {b3, b4, b5}]);
f[5, {1, b2_, b3_, b4_, b5_}] :=  f[5, {1, b2, b3, b4, b5}] = 
   f[5, {1, b2 - 2, b3, b4, b5}] - z (f[5, {1, b2 - 1, b3, b4, b5}] 
    - a^b2 f[3, {b3, b4, b5}]);
f[5, {2, b2_, b3_, b4_, b5_}] :=   f[5, {2, b2, b3, b4, b5}] = 
   a^b2 f[3, {b3, b4, b5}] - z (a^(-1) f[5, {1, b2 - 1, b3, b4, b5}] 
     - f[5, {1, b2, b3, b4, b5}]);
f[5, {b1_, 1, b3_, b4_, b5_}] :=  f[5, {b1, 1, b3, b4, b5}] = 
   f[5, {b1 - 2, 1, b3, b4, b5}] 
    - z (a^(1 - b1) f[3, {b3 + 1, b4, b5}]  
     - f[5, {b1 - 1, 1, b3, b4, b5}]);
f[5, {b1_, b2_, b3_, b4_, b5_}] :=  f[5, {b1, b2, b3, b4, b5}] = 
  f[5, {b1 - 2, b2, b3, b4, b5}]  
   - z (a^(-b1 + 1) f[5, {1, b2 - 1, b3, b4, b5}] 
     - f[5, {b1 - 1, b2, b3, b4, b5}])

f[i_, {1, 1, b3_, 1, b5_, btail__}] :=  
f[i, {1, 1, b3, 1, b5, btail}] =  a^(-b3) f[i - 4, {1 + b5, btail}]  
  - z (f[i - 2, {b3 + 1, 1, b5, btail}] 
    - a f[i - 2, {b3, 1, b5, btail}]);
f[i_, {1, 1, b3_, b4_, btail__}] :=  f[i, {1, 1, b3, b4, btail}] = 
   a^(-b3) f[i - 2, {1, b4 - 1, btail}] 
    - z (f[i - 2, {b3 + 1, b4, btail}] 
     - a f[i - 2, {b3, b4, btail}]);
f[i_, {2, 1, b3_, btail__}] :=  f[i, {2, 1, b3, btail}] = 
   a f[i - 2, {b3, btail}] - z (a^(-1) f[i - 2, {b3 + 1, btail}] 
    - f[i, {1, 1, b3, btail}]);
f[i_, {1, 2, b3_, btail__}] :=  f[i, {1, 2, b3, btail}] = 
   f[i - 2, {b3 + 1, btail}] - z (f[i, {1, 1, b3, btail}] 
    - a^2 f[i - 2, {b3, btail}]);
f[i_, {1, b2_, btail__}] := f[i, {1, b2, btail}] = 
   f[i, {1, b2 - 2, btail}] -  z (f[i, {1, b2 - 1, btail}] 
     - a^b2 f[i - 2, {btail}]);
f[i_, {2, b2_, btail__}] := f[i, {2, b2, btail}] = 
   a^b2 f[i - 2, {btail}]  
    - z (a^(-1) f[i, {1, b2 - 1, btail}] - f[i, {1, b2, btail}]);
f[i_, {b1_, 1, b3_, btail__}] :=   f[i, {b1, 1, b3, btail}] = 
   f[i, {b1 - 2, 1, b3, btail}] 
    - z (a^(1 - b1) f[i - 2, {b3 + 1, btail}]  
     -  f[i, {b1 - 1, 1, b3, btail}]);
f[i_, {b1_, b2_, btail__}] :=  f[i, {b1, b2, btail}] = 
  f[i, {b1 - 2, b2, btail}] 
   - z (a^(-b1 + 1) f[i, {1, b2 - 1, btail}] 
     - f[i, {b1 - 1, b2, btail}])
\end{verbatim}

\begin{verbatim}
c := a;
e := a^(-1);
d := z^(-1) a - z^(-1) a^(-1) + 1 - z a^(-1) + z a;
\end{verbatim}

The code for negative $b_i$'s is given below.

\begin{verbatim}
Clear[f, a, c, d, b1, z, b2, b3, b4, b5, btail, i]
f[1, {-1}] = e;
f[1, {-2}] = d;
f[1, {b1_}] := 
 f[1, {b1}] = f[1, {b1 + 2 }] + z (a^(-1 - b1) - f[1, {b1 + 1}])

f[2, {-1, -1}] := d;
f[2, {-1, -2}] := f[2, {-1, -2}] = a + z (d - a^(-2));
f[2, {-2, -1}] := f[2, {-2, -1}] = a^(-1) 
	+ z (a^2 - f[2, {-1, -1}]);
f[2, {-1, b2_}] := f[2, {-1, b2}] = 
f[2, {-1, b2 + 2}] + z (f[2, {-1, b2 + 1}] - a^b2);
f[2, {-2, b2_}] := f[2, {-2, b2}] = 
a^b2 + z (a f[2, {-1, b2 + 1}] - f[2, {-1, b2}]);
f[2, {b1_, -1}] := f[2, {b1, -1}] = f[1, {b1 - 1}];
f[2, {b1_, b2_}] :=  f[2, {b1, b2}] = 
  f[2, {b1 + 2, b2}] + z (a^(-1 - b1) f[2, {-1, b2 + 1}] 
   - f[2, {b1 + 1, b2}])

f[3, {-1, -1, b3_}] :=   f[3, {-1, -1, b3}] = 
   a^(-b3) + z (f[1, {b3 - 1}] - a^(-1) f[1, {b3}]);
f[3, {-1, -2, b3_}] :=  f[3, {-1, -2, b3}] = 
   f[1, {b3 - 1}] + z (f[3, {-1, -1, b3}] - a^(-2) f[1, {b3}]);
f[3, {-2, -1, b3_}] :=  f[3, {-2, -1, b3}] = 
   a^(-1) f[1, {b3}] + z (a f[1, {b3 - 1}] - f[3, {-1, -1, b3}]);
f[3, {-1, b2_, b3_}] :=  f[3, {-1, b2, b3}] = 
   f[3, {-1, b2 + 2, b3}] + z (f[3, {-1, b2 + 1, b3}] 
    - a^b2 f[1, {b3}]);
f[3, {-2, b2_, b3_}] := f[3, {-2, b2, b3}] = 
   a^b2 f[1, {b3}] + z (a f[3, {-1, b2 + 1, b3}] 
    - f[3, {-1, b2, b3}]);
f[3, {b1_, -1, b3_}] :=  f[3, {b1, -1, b3}] = 
   f[3, {b1 + 2, -1, b3}] + z (a^(-1 - b1) f[1, {b3 - 1}] 
    - f[3, {b1 + 1, -1, b3}]);
f[3, {b1_, b2_, b3_}] :=  f[3, {b1, b2, b3}] = 
  f[3, {b1 + 2, b2, b3}] + z (a^(-b1 - 1) f[3, {-1, b2 + 1, b3}] 
   - f[3, {b1 + 1, b2, b3}])

f[4, {-1, -1, b3_, -1}] :=  f[4, {-1, -1, b3, -1}] = 
   a^(1 - b3) + z (f[2, {b3 - 1, -1}] - a^(-1) f[2, {b3, -1}]);
f[4, {-1, -1, b3_, b4_}] :=   f[4, {-1, -1, b3, b4}] = 
   a^(-b3) f[2, {-1, b4 + 1}] + z (f[2, {b3 - 1, b4}] 
    - a^(-1) f[2, {b3, b4}]);
f[4, {-2, -1, b3_, b4_}] :=   f[4, {-2, -1, b3, b4}] = 
   a^(-1) f[2, {b3, b4}] + z (a f[2, {b3 - 1, b4}] 
   - f[4, {-1, -1, b3, b4}]);
f[4, {-1, -2, b3_, b4_}] :=  f[4, {-1, -2, b3, b4}] = 
   f[2, {b3 - 1, b4}] + z (f[4, {-1, -1, b3, b4}] 
    - a^(-2) f[2, {b3, b4}]);
f[4, {-1, b2_, b3_, b4_}] :=  f[4, {-1, b2, b3, b4}] = 
   f[4, {-1, b2 + 2, b3, b4}] 
    + z (f[4, {-1, b2 + 1, b3, b4}] - a^b2 f[2, {b3, b4}]);
f[4, {-2, b2_, b3_, b4_}] :=  f[4, {-2, b2, b3, b4}] = 
   a^b2 f[2, {b3, b4}] + z (a f[4, {-1, b2 + 1, b3, b4}] 
    - f[4, {-1, b2, b3, b4}]);
f[4, {b1_, -1, b3_, b4_}] := f[4, {b1, -1, b3, b4}] = 
  f[4, {b1 + 2, -1, b3, b4}] 
   + z(a^(-1 - b1) f[2, {b3 - 1, b4}]
    - f[4, {b1 + 1, -1, b3, b4}]);
f[4, {b1_, b2_, b3_, b4_}] :=  f[4, {b1, b2, b3, b4}] = 
  f[4, {b1 + 2, b2, b3, b4}]  
   + z (a^(-b1 - 1) f[4, {-1, b2 + 1, b3, b4}]  
    - f[4, {b1 + 1, b2, b3, b4}])

f[5, {-1, - 1, b3_, -1, b5_}] :=  f[5, {-1, -1, b3, -1, b5}] = 
   a^(-b3) f[1, {b5 - 1}] + z (f[3, {b3 - 1, -1, b5}] 
    - a^(-1) f[3, {b3, -1, b5}]);
f[5, {-1, -1, b3_, b4_, b5_}] := f[5, {-1, -1, b3, b4, b5}] = 
   a^(-b3) f[3, {-1, b4 + 1, b5}] + z (f[3, {b3 - 1, b4, b5}] 
    - a^(-1) f[3, {b3, b4, b5}]);
f[5, {-2, -1, b3_, b4_, b5_}] :=  f[5, {-2, -1, b3, b4, b5}] = 
   a^(-1) f[3, {b3, b4, b5}] + z (a f[3, {b3 - 1, b4, b5}] 
    - f[5, {-1, -1, b3, b4, b5}]);
f[5, {-1, -2, b3_, b4_, b5_}] :=  f[5, {-1, -2, b3, b4, b5}] = 
   f[3, {b3 - 1, b4, b5}] 
   + z (f[5, {-1, -1, b3, b4, b5}] - a^(-2) f[3, {b3, b4, b5}]);
f[5, {-1, b2_, b3_, b4_, b5_}] :=  f[5, {-1, b2, b3, b4, b5}] = 
   f[5, {-1, b2 + 2, b3, b4, b5}] + z (f[5, {-1, b2 + 1, b3, b4, b5}] 
    - a^b2 f[3, {b3, b4, b5}]);
f[5, {-2, b2_, b3_, b4_, b5_}] :=  f[5, {-2, b2, b3, b4, b5}] = 
   a^b2 f[3, {b3, b4, b5}] + z (a f[5, {-1, b2 + 1, b3, b4, b5}]
    - f[5, {-1, b2, b3, b4, b5}]);
f[5, {b1_, -1, b3_, b4_, b5_}] := f[5, {b1, -1, b3, b4, b5}] = 
  f[5, {b1 + 2, -1, b3, b4, b5}] 
   + z (a^(-1 - b1) f[3, {b3 - 1, b4, b5}]  
    - f[5, {b1 + 1, -1, b3, b4, b5}]);
f[5, {b1_, b2_, b3_, b4_, b5_}] :=  f[5, {b1, b2, b3, b4, b5}] = 
  f[5, {b1 + 2, b2, b3, b4, b5}] 
  + z (a^(-b1 - 1) f[5, {-1, b2 + 1, b3, b4, b5}] 
   - f[5, {b1 + 1, b2, b3, b4, b5}])

f[i_, {-1, -1, b3_, -1, b5_, btail__}] := 
f[i, {-1, -1, b3, -1, b5, btail}]= a^(-b3) f[i - 4, {b5 - 1, btail}] 
 + z (f[i - 2, {b3 - 1, -1, b5, btail}] 
  - a^(-1) f[i - 2, {b3, -1, b5, btail}]);
f[i_, {-1, -1, b3_, b4_, btail__}] :=  f[i, {-1, -1, b3, b4, btail}] = 
  a^(-b3) f[i - 2, {-1, b4 + 1, btail}] 
   + z (f[i - 2, {b3 - 1, b4, btail}] 
    - a^(-1) f[i - 2, {b3, b4, btail}]);
f[i_, {-2, -1, b3_, btail__}] := f[i, {-2, -1, b3, btail}] = 
   a^(-1) f[i - 2, {b3, btail}] + z (a f[i - 2, {b3 - 1, btail}] 
    - f[i, {-1, -1, b3, btail}]);
f[i_, {-1, -2, b3_, btail__}] :=  f[i, {-1, -2, b3, btail}] = 
   f[i - 2, {b3 - 1, btail}] + z (f[i, {-1, -1, b3, btail}] 
    - a^(-2) f[i - 2, {b3, btail}]);
f[i_, {-1, b2_, btail__}] :=   f[i, {-1, b2, btail}] = 
   f[i, {-1, b2 + 2, btail}] + z (f[i, {-1, b2 + 1, btail}] 
    - a^b2 f[i - 2, {btail}]);
f[i_, {-2, b2_, btail__}] :=   f[i, {-2, b2, btail}] = 
   a^b2 f[i - 2, {btail}] 
    + z (a f[i, {-1, b2 + 1, btail}] - f[i, {-1, b2, btail}]);
f[i_, {b1_, -1, b3_, btail__}] :=   f[i, {b1, -1, b3, btail}] = 
  f[i, {b1 + 2, -1, b3, btail}] 
   + z (a^(-1 - b1) f[i - 2, {b3 - 1, btail}] 
    - f[i, {b1 + 1, -1, b3, btail}]);
f[i_, {b1_, b2_, btail__}] :=  f[i, {b1, b2, btail}] = 
  f[i, {b1 + 2, b2, btail}] + z (a^(-b1 - 1) f[i, {-1, b2 + 1, btail}] 
   - f[i, {b1 + 1, b2, btail}]).

\end{verbatim}


\end{document}